\documentclass[12pt]{article}
\usepackage{amsmath}
\usepackage[dvipsnames,usenames]{color}
\usepackage{amsfonts}
\usepackage{mathrsfs}
\usepackage{amssymb}
\usepackage{amsthm}
\def\<{\langle}
\def\>{\rangle}
\numberwithin{equation}{section}
\def\<{\langle}
\def\>{\rangle}

\def\ra{\rightarrow}
\def\p{\partial}
\def\a{\alpha}

\def\w{\widetilde}

\def\BB{{\mathbb B}}
\def\RR{{\mathbb R}}
\def\CC{{\mathbb C}}

\def\ZZ{{\mathbb Z}}
\def\HH{{\mathbb H}}

\def\ov{\overline}

\def\b{\beta}
\def\a{\alpha}

\def\wt{\widetilde}
\def\ra{\rightarrow}

\def\a{\alpha}

\def\b{\beta}

\def\a{\alpha}

\def\a{\alpha}

\newtheorem{theorem}{Theorem}[section]
\newtheorem{lemma}[theorem]{Lemma}
\newtheorem{Corollary}[theorem]{Corollary}

\newtheorem{conjecture}[theorem]{Conjecture}
\newtheorem{example}[theorem]{Example}

\date{\ }

\begin{document}
\title{\bf On the Third Gap
  for Proper Holomorphic Maps between Balls}

\author{{Xiaojun  Huang}\footnote{
Supported in part by DMS-1101481}\qquad  Shanyu Ji \qquad Wanke
Yin\footnote{Supported in part by  FANEDD-201117, ANR-09-BLAN-0422,
RFDP-20090141120010, NSFC-10901123 and NSFC-11271291.}}

  \maketitle








\section{Introduction}
Write $\mathbb{B}^{n}$ for the unit ball in the complex space
$\CC^{n}$. Recall that a holomorphic map $F$ from ${\BB}^n$ into
${\BB}^{N}$ is called proper if for any compact subset $K\subset
{\BB}^N$, $F^{-1}(K)$ is also a compact subset in ${\BB}^n$. A
holomorphic map defined over ${\BB}^n$ is said to be rational if it
can be written as $\frac{P}{q}$ with $P$ a holomorphic polynomial
map and $q$ a holomorphic polynomial function. This paper continues
the recent work  in  Hamada \cite{Ha1}, Huang-Ji-Xu \cite{HJX1},
etc. Our main purpose is to prove the following gap rigidity
theorem:

\begin{theorem} \label{thm1} Let $F$ be a proper  rational map from $\BB^{n}$
into $\BB^{N}$ with $n> 7$ and  $3n+1\le N\le 4n-7$.
Then there is an automorphism $\tau\in Aut({\BB}^N)$ such that
$\tau\circ F=(G,0')=(G,0,0,\cdots,0)$, where  $G$ is a  proper
holomorphic  rational map from ${\BB}^n$ into ${\BB}^{3n}$.
\end{theorem}

Theorem \ref{thm1} roughly says that there is no new proper rational
map added for $N$ in the closed interval denoted by ${\cal
I}_3:=[3n+1,4n-7]$. The following example shows that Theorem
\ref{thm1} is sharp. (See Remark A in $\S 5$ for more discussions on
this example.)
\begin{example} For  $n\ge 2,\ \lambda,\mu\in (0,1)$, define the proper monomial   map $F$ from
${\BB}^n$ into ${\BB}^{3n}$ as follows:
\begin{equation}
\label{example} F=\big(z_1,\cdots,z_{n-2}, \lambda z_{n-1},z_n,
\sqrt{1-\lambda^2} z_{n-1}(z_1,\cdots,z_{n-1},\mu z_n,
\sqrt{1-\mu^2} z_nz)\big).
\end{equation}
For such a map $F$, there is no  $\tau\in Aut({\BB}^{3n})$ such that
$\tau\circ F=(G,0').$
Also,  there are  proper monomial maps $F$ from ${\BB}^n$ into
${\BB}^{4n-6}$ (\cite{HJY}) such that for any $\tau\in
Aut({\BB}^{4n-6})$, $\tau\circ F$ can not be of the form $(G,0')$.
\end{example}
The rationality theorem proved in \cite{Hu2} \cite{HJX2} says that any proper
holomorphic map from ${\BB}^n$ into ${\BB}^N$ with $N\le n(n+1)/2$,
that is three times differentiable up to the boundary, must be
rational. Hence, Theorem \ref{thm1} can be stated in the following
more general form:

\begin{theorem} \label{thmmm1} Let $F$ be a proper holomorphic  map from $\BB^{n}$
into $\BB^{N}$ with $n > 7$ and  $3n+1\le N\le 4n-7$. Assume that
$F$ is $C^3$-smooth up to the boundary.
Then there is an automorphism $\tau\in Aut({\BB}^N)$ such that
$\tau\circ F=(G,0')$, where  $G$ is a  proper  rational map from
${\BB}^n$ into ${\BB}^{3n}$.
\end{theorem}

Rigidity property is a fundamental property for holomorphic
functions with several variables. The study of various rigidity
properties for proper holomorphic maps between balls in complex
Euclidean spaces goes back  to the pioneer  paper of Poincar\'e
\cite{Po}. Since
 then,  much attention has been paid to such an investigation.
 When $n>1$,
 a  result of  Alexander \cite{Alx} states that any
 proper holomorphic
self-map of the unit ball ${\BB}^n$ in ${\CC}^{n}$ with $n>1$ is an
automorphism.
 Recall that two proper holomorphic maps $f,g$ from
 ${\BB}^n$ into ${\BB}^{N}$
 are said to be equivalent if there are
$\sigma\in \hbox {Aut}( {\BB}^n)$ and $\tau\in \hbox {Aut}({\BB}^N)$
such that $g=\tau\circ f\circ \sigma$.  A proper holomorphic map from
 ${\BB}^n$ into ${\BB}^{N}$ is said to be linear or totally geodesic if
 it is equivalent to the standard big circle embedding $L(z):z\ra
(z,0)$.
 Webster in \cite{W} considered the
geometric structure of proper holomorphic maps between balls in
 complex spaces of different dimensions.
 He showed that a proper
holomorphic map from ${\BB}^n$ into ${\BB}^{n+1}$ with $n>2$,
 which is three
times differentiable up to the boundary, is a totally geodesic
embedding. Subsequently, Cima-Suffridge \cite{CS1} reduced the boundary
regularity in Webster's theorem  to the $C^2$-regularity. Motivated
by a conjecture  in \cite{CS1},  Faran in \cite{Fa2} showed that any proper
holomorphic map from ${\BB}^n$ into $ {\BB}^{N}$ with $N<2n-1$, that
is real analytic up to the boundary, is  a totally geodesic
embedding. Forstneric in \cite{Fo1} proved that
 any proper holomorphic map from ${\BB}^n$ into ${\BB}^{N}$ is
rational, if the map is $C^{N-n+1}$-regular up to the boundary,
which, in particular, reduces the regularity assumption in Faran's
linearity theorem to the $C^{N-n+1}$-smoothness. In a paper of Mir
\cite{Mir}, the theorem of Forstneric was weakened to the case where the
source manifold needs only to be  assumed to be a real analytic
hyper-surface. See also a related paper by Baouendi-Huang-Rothschild
\cite{BHR} and a later generalization in Meylan-Mir-Zaitsve [MMZ].  At
this point, we mention that the discovery of inner functions can be
used to show that there is a
 proper holomorphic map from ${\BB}^n$ into ${\BB}^{n+1}$, which can not be $C^2$-smooth
 at any boundary point. (See \cite{HS}, \cite{Low}, \cite{Fo2}, \cite{Ste}, etc).

Write ${\cal I}_1=[n+1,2n-2]$. The aforementioned theorem of Faran
says that there is no new proper rational map added when the target
dimension $N\in {\cal I}_1$. We call ${\cal I}_1$ the first gap
interval for proper holomorphic mappings between balls. In \cite{Fa1},
Faran showed that there are four   different inequivalent proper
holomorphic maps from ${\BB}^2$ into ${\BB}^{3}$,
 which are $C^3$-smooth
up to the boundary. However, the only embeddings  are linear maps.


In \cite{Hu1} and, subsequently, \cite{HJ},  two  questions arising from the
above mentioned work were considered. In \cite{Hu1}, the first author
proved that any proper holomorphic map which is only $C^2$-regular
up to the boundary must be linear if $N<2n-1$, by applying a very
different method from the previous work, answering a long standing
open question in the field (see \cite{CS1} \cite{Fo2}). While it has been open
for many years to answer if the $C^1$-boundary regularity is still
enough for this super-rigidity to hold, the result in [Hu1] gives a
first result in which the required regularity is independent of the
codimension. In [Theorem 1, Theorem 2.3; \cite{HJ}] [Corollary 2.1, \cite{Hu3}],
it was shown that any proper holomorphic map from ${\BB}^n$ into
${\BB}^N$ with $N=2n-1, n\ge 3$, which is $C^2$-smooth up to the
boundary, is either linear or equivalent to the Whitney map
\begin{equation}\label{000}
 W:
z=(z_1,\cdots,z_n)=(z',z_n)\ra (z_1,\cdots,z_{n-1},
z_nz)=(z',z_nz).
\end{equation}
 Since the Whitney map is not an
immersion, together with the aforementioned work of Faran \cite{Fa1},
this shows that any proper holomorphic {\it embedding} from
${\BB}^n$ into ${\BB}^N$ with $N=2n-1$, which is twice continuously
differentiable up to the boundary, must be a linear map. Earlier,
D'Angelo constructed the following family $F_\theta$ of mutually
inequivalent proper quadratic   monomial maps from ${\BB}^n$ into
${\BB}^{2n}$ (See \cite{DA}):
\begin{equation}
F_{\theta}(z',z_n)=(z',(\cos\theta)z_n,(\sin\theta)z_1z_n, \cdots,
(\sin\theta)z_{n-1}z_n, (\sin\theta)z_n^2),\ 0< \theta\leq \pi/2.
\label{eqn:ftheta}
\end{equation}
Notice that by  adding $N-2n$ zero components to  the D'Angelo map
$F_\theta$,  we get  a proper monomial embedding from ${\BB}^n$ into
${\BB}^{N}$ for any $N\ge 2n$.  The combining effort in \cite{Fa2} and
\cite{HJ} gives a complete description to the linearity problem for
proper holomorphic embeddings from ${\BB}^n$ into ${\BB}^N$, which
are $C^2$-smooth up to the boundary.
However, in  applications,  one  still hopes to get the linearity
for mappings with a rich geometric structure.
For instance,  the following difficult  problem initiated from the work of Siu, Mok
\cite{Mok} and others has been open for more than thirty years: (See
Cao-Mok \cite{CMk} for the work when $N\le 2n-1$.)
\begin{conjecture} (Siu, Mok): {\it Let $f$ be a proper holomorphic
mapping from ${\BB}^n$ into ${\BB}^N$ with $ 1<n<N$. Write
$M=F({\BB}^n)$. Suppose that there is a subgroup $\Gamma$ of
$Aut({\BB}^N)$ such that (1). for any $\sigma \in \Gamma$,
$\sigma(M)=M$; (2) $M/\Gamma$ is compact. Then $f$ is  a linear
embedding.} \end{conjecture}


In a  recent  paper of Hamada \cite{Ha1}, based on a careful
analysis on the Chern-Moser normal form  method as developed in
[Hu1] and [HJ], it was proved that all proper rational  maps from
${\BB}^n$ into ${\BB}^{2n}$ with $n\ge 4$ are either equivalent to
the Whitney map $W$ in (\ref{000}) or  the D'Angelo map $F_\theta$.
After the work of Hamada  \cite{Ha1}, the first two authors and Xu
in \cite{HJX1} proved that a proper holomorphic map from ${\BB}^n$
into ${\BB}^N$ with $4\le n\leq N \leq 3n-4$, that is $C^3$-smooth
up to the boundary, is equivalent to either the map $(W,0')$ or
$(F_\theta,0')$ with $\theta\in [0,\pi/2)$.
An immediate consequence of the work in [HJX1] is that there is no
new map added  when $N\in {\cal I}_2$ with ${\cal
I}_2:=[2n+1,3n-4]$. Since there are proper  monomial maps from
${\BB}^n$ into ${\BB}^N$ for $3n-3\le N\le 3n$ or $2n-1\le N\le 2n$,
that are not equivalent to maps of the form $(G,0')$, we call ${\cal
I}_2$ the second gap interval for proper holomorphic maps between
balls.

By \cite{HJY}, for any $N$ with  $3n-3\le N\le 3n$ or $4n-6\le N\le 4n$,
there are many proper  monomial  maps from ${\BB}^n$ into ${\BB}^N$,
that are not equivalent to  maps of the form $(G,0')$. Theorem
\ref{thm1} in the present paper thus provides a third gap interval
${\cal I}_3:=[3n+1,4n-7]$ for proper holomorphic maps between balls.

More generally, for any $n\ge 3$, write $K(n)$ for  the largest
positive integer $m$ such that $m(m+1)/2<n$. Then
$K(n)=[\frac{-1+\sqrt{1+8n}}{2}]$ if $\frac{-1+\sqrt{1+8n}}{2}$ is
not an integer; and $K(n)=\frac{-1+\sqrt{1+8n}}{2}-1$, otherwise.
For each $1\le k\le K(n)$, define ${\cal I}_k:
=[kn+1,(k+1)n-\frac{k(k+1)}{2}-1]$. Then ${\cal I}_k$ is a closed
interval containing positive integers if $n\ge 2+\frac{k(k+1)}{2}$.
Apparently, ${\cal I}_{k}\cap{\cal I}_{k'}=\emptyset$ for $k\not
=k'$; and ${\cal I}_k$ for $k=1,2,3$ are exactly the same intervals
defined above. Write ${\cal I}=\cup _{k=1}^{K(n)}{\cal I}_k$. Then,
for $$\max_{N \in {\cal
I}}N=(K(n)+1)n-\frac{K(n)(K(n)+1)}{2}-1\approx
\frac{-1+\sqrt{1+8n}}{2}n-n-1\approx \sqrt{2}n^{\frac{3}{2}}-n-1.$$
 For any $N\not \in {\cal I}$ (which certainly is the case when
 $ N\ge 1.42n^{\frac{3}{2}}$), by not a complicated construction,
the authors obtained in [HJY] many monomial proper holomorphic maps
from ${\BB}^n$ into ${\BB}^N$, that can not be equivalent to maps of
the form $(G,0')$. ( See Theorem 2.8, \cite{HJY}). Earlier in
\cite{DL}, for $N\ge n^2-2n+2$, D'Angelo and Lebl, by a different
method, constructed a proper monomial map from ${\BB}^n$ into
${\BB}^N$, that is not equivalent to a map of the form $(G,0')$.
However, we have  not been able to find a map, not equivalent to a
map of the form $(G,0')$, for $N\in {\cal I}$. Indeed, the first,
the second and the third gap intervals mentioned above suggest the
following
 conjecture:

\begin{conjecture}\label{conj1} (Huang-Ji-Yin \cite{HJY}) Let $n\ge 3$ be
a positive integer, and let ${\cal I}_k$ ($1\le k\le K(n)$) be
defined above.
 Then
 any  proper
holomorphic rational  map $F$ from ${\BB}^n$ into ${\BB}^N$  is
equivalent to a map of the form $(G,0')$ if and only if $N\in {\cal
I}_k$ for some $1\le k\le K(n)$.
\end{conjecture}
As mentioned above, the ``$\Longrightarrow$" part follows from
Theorem 2.8 of \cite{HJY}; also   the
 conjecture holds for $k=1,2,3$. An affirmative solution to this gap conjecture would
tells exactly for what pair $(n,N)$ there are no new  proper
rational maps  added.

Next, we describe briefly the idea for the proof of Theorem
\ref{thm1}. The proofs for the first and the second gaps are
immediate applications of the much more precise classification
results. When $N\in {\cal I}_3$,  making a precise classification
for all maps seems to be hard. We need a different approach from the
work  in Huang-Ji \cite{HJ}, Hamada \cite{Ha1} and Huang-Ji-Xu \cite{HJX1}.
Consider the setting in the Heisenberg hypersurface case.
 Let $F$ be a holomorphic map defined near
$0$ with $F(0)=0$ into ${\CC}^N$. Then the Taylor formula says that
$F(z)=\sum_{\a}\frac{D^\a F}{\a!}(0)z^\a.$ Hence the image of $F$
stays in the linear subspace spanned by $\{D^\a F(0)\}_\a$. If
$spann\{D^\a F(0)\}_\a \not ={\CC}^N,$ we get a gap from  $F$. The
crucial point in our argument is to find, for our map, a basis of
$spann\{D^\a F(0)\}_\a$. The way to achieve is to get a good normal
form for $F$. However, this is a highly non-linear normalization
problem, for the maps need to satisfy the fundamental non-linear
equation. While it is easy to get linear independent set from the
first and the second jets, finding more linearly independent
elements to form a basis from the higher order jets is very
involved. The basic tool at our disposal for this approach  is a
 lemma of the first author proved in [Lemma 3.2, \cite{Hu1}]. For
$N\in {\cal I}_3$, it turns out that there is only one more linearly
independent element for the map from the higher order jets.
For the study of general but very rough jet determination problems
for holomorphic maps, there has been much work done in the past. We
refer the reader to the book by Baouendi-Ebenfelt-Rothschild \cite{BER}
and a  paper by Lamel-Mir \cite{LM}. However, what we need here is a very
precise  jet determination, which is only doable due to the extra
geometric structure for the maps in our setting.

It appears to us that a fundamental fact which dominates the gap
rigidity for holomorphic maps between balls is [Lemma 3.2,
\cite{Hu1}].  In the course of the proof our main theorem, one finds
that the assumption $N\in {\cal I}_3$ is exactly what is needed, in
several induction steps, for applying [Lemma 3.2, \cite{Hu1}]. We
hope that the method of the present paper may motivate the general
study of Conjecture \ref{conj1}.

Our discussion above only touches the linearity and the gap rigidity
part from a vast amount of work for mappings between balls.
We would like to mention that
 there has been a lot of interesting   work done in the past on the study of proper monomial
maps between balls by D'Angelo and his coauthors. (See the book of
D'Angelo \cite{DA} for many references therein.) Here, we  mention,
in particular, two   papers on the   degree estimates for proper
monomial maps by D'Angelo-Kos-Riehl \cite{DKR} and Lebl-Peters
\cite{LP}. The study for mappings between balls is also related to
the problem of decomposing a positive Hermitian form into the sum
square of holomorphic functions, for which we refer the reader to a
recent survey article by Putinar \cite{Put} as well as many
references therein. Here, we just mention a result obtained by
Quillen-Catlin-D'Angelo  in \cite{Qu}  and \cite{CD}, which states
that for any positive bi-homogenous polynomial $H(z,\ov{z})$, there
is a sufficiently large integer $N$ such that
$|z|^{2N}H(z,\ov{z})=\sum_{j=1}^{N'}|h_j(z)|^2$ with $h_j(z)$
holomorphic  polynomials. This has an immediate consequence (see
\cite{CD}) that for any homogenous polynomial map $q(z)$ into
${\CC}^N$ with $
 |q(z)|<1$ on the sphere ,  there exists a vector valued polynomial $p(z)$ with $N(q)$-
 components such that ${( q(z), p(z))}$ properly holomorphically
maps $\mathbb{ B}^{n}$ into $\mathbb{ B}^{N+N(q)}$, where $N(q)$
depends on $q$ and the value $1-|q(z)|^2$ and could be very large.



\section{Notations and Preliminaries}
In this section, we set up notation and recall a result established
in Huang-Ji-Xu \cite{HJX1} and a lemma from \cite{Hu1} which will be crucial
for our proof of Theorem \ref{thm1}.

  Write $\HH_n:=\{(z,w)\in {\CC}^{n-1}\times {\CC}:
  \ \hbox{Im}(w)>|z|^2\}$ for the Siegel upper-half space.
  Similarly, we can define the notion of  proper rational maps from ${\HH}_n$ into
  ${\HH}_N$.
  Since the Cayley transformation
\begin{equation}
\rho_n: {\HH}_n\to {\BB}^n, \ \ \rho_n(z,w)=\bigg(\frac{2z}{1-iw},\
\frac{1+iw}{1-iw}\bigg) \label{eqn:rho}
\end{equation}
is a biholomorphic mapping between $\HH_n$ and ${\BB}^n$,
  we can identify a proper rational map $F$ from ${\BB}^n$ into ${\BB}^N$
  %
  with
$\rho^{-1}_N\circ F\circ\rho_n$, which is a proper rational map from
${\HH}_n$ into ${\HH}_N$.
 By a well-known result of Cima-Suffridge \cite{CS2}, $F$ extends holomorphically across
 the boundary $\p {\BB}^n$.

Parameterize $\partial \HH_n$ by $(z,\overline{z},u)$ through the
map $(z,\overline{z},u)\to (z,u+i|z|^2)$. In what follows, we will
assign the weight of $z$ and $u$ to be  $1$ and $2$, respectively.
For a non-negative integer $m$, a function $h(z,\overline{z},u)$
defined over a small ball  $U$ of $0$ in $\partial \HH_n$  is said
to be of quantity $o_{wt}(m)$ if
$\frac{h(tz,t\overline{z},t^2u)}{|t|^{m}}\to 0$ uniformly for
$(z,u)$ on any
  compact subset of $U$ as $t(\in {\RR})\to 0$.
We use the notation $h^{(k)}$ to denote a polynomial $h$ which has
  weighted degree $k$. Occasionally, for a holomorphic function (or
  map) $H(z,w)$, we write $H(z,w)=\sum_{k,l=0}^{\infty}H^{(k,l)}(z)w^l$
  with  $H^{(k,l)}(z)$ a polynomial of degree $k$ in $z$.
\bigskip

Let $F=(f,\phi,g)=(\widetilde{f}, g)= (f_1,\cdots,f_{n-1},
\phi_1,\cdots, \phi_{N-n},g)$ be a non-constant $C^2$-smooth CR map
from $\partial{\HH}_n$ into $\partial{\HH}_N$ with $F(0)=0$. For
each $p=(z_0, w_0)\in M$ close to $0$, we write $\sigma^0_p\in
\hbox{Aut}(\HH_n)$ for the map sending $(z,w)$ to $(z+z_0, w+w_0+2i
\langle z,\overline{z_0} \rangle )$ and
$\tau^F_p\in\hbox{Aut}(\HH_N)$ by defining
$$\tau^F_p(z^*,w^*)=(z^*-\widetilde{f}(z_0,w_0),w^*-\overline{g(z_0,w_0)}-
2i \langle z^*,\overline{\widetilde{f}(z_0,w_0)} \rangle ).$$
  Then $F$ is
equivalent to

\begin {equation}
F_p=\tau^F_p\circ F\circ \sigma^0_p=(f_p,\phi_p,g_p).
\label{eqn:nor01}
\end{equation}

  Notice that
$F_0=F$ and $F_p(0)=0$. The following is fundamentally important for
the understanding of the geometric properties of $F$.

\medskip
{\bf Lemma 2.1} ([$\S 2$, Lemma 5.3, \cite{Hu1}):
    Let $F$ be a $C^2$-smooth CR map
from  $\partial\HH_n$ into $\partial\HH_N$, $2\le n\le N$. For each
$p\in \partial \HH_n$, there is an automorphism $\tau^{**}_p\in
  {Aut}_0({\HH}_N)$ such that
$F_{p}^{**}:=\tau^{**}_p\circ F_p$ satisfies the following
normalization:
$$f^{**}_{p}=z+{\frac{i}{ 2}}a^{**(1)}_{p}(z)w+o_{wt}(3),\ \phi_p^{**}
={\phi_p^{**}}^{(2)}(z)+o_{wt}(2), \ g^{**}_{p}=w+o_{wt}(4),\
\hbox{with}$$
$$\langle \overline{z}, a_{p}^{**(1)}(z)\rangle
|z|^2=|{\phi_p^{**}}^{(2)}(z)|^2.$$

{\bf Definition 2.2} ([Hu2]) Write
$\mathcal{A}(p)=-2i(\frac{\partial^2(f^{\ast\ast}_p)
  _l}{\partial z_j\partial w}|_0)_{1\leq j,l\leq (n-1)}$ in the above
lemma. We call the rank
  of the $(n-1)\times (n-1)$ matrix  $\mathcal{A}(p)$, which we denote by
  $Rk_F(p)$, the {\it geometric rank} of $F$ at $p$.

\medskip
Define the {\it geometric rank} of $F$ to be $\kappa_0(F)=max_{p\in
  \partial\HH_n} Rk_F(p)$.
  Define the
geometric rank of a proper holomorphic map ${\BB}^n$ into ${\BB}^N$,
that is $C^2$-smooth up to the boundary, to be the one for the map
$\rho_N^{-1}\circ F \circ \rho_n.$
By \cite{Hu2}, $\kappa_0(F)$ depends only on the equivalence class of $F$
and when $N<\frac{n(n+1)}{2}$,  $\kappa_0(F)\le n-2$.
In \cite{HJX1}, the authors proved the following normalization theorem
for maps with geometric rank bounded by $n-2$, though only part of
it is needed later:

\begin{theorem}\label{thm3} (\cite{HJX1})
Suppose that $F$ is a rational proper holomorphic map from ${\HH}_n$
into ${\HH}_N$, which has
 geometric rank  $1\le\kappa_0\le n-2$ with $F(0)=0$. Then there are
$\sigma\in \hbox{Aut}({\HH}_n)$ and
  $\tau\in \hbox{Aut}({\HH}_N)$ such that
  $\tau\circ F\circ \sigma $ takes
the following form, which is still denoted by $F=(f,\phi,g)$ for
convenience of notation:

\begin{equation}
\left\{
  \begin{array}{l}
  f_l=\sum_{j=1}^{\kappa_0}z_jf_{lj}^*(z,w),\ \ l\le\kappa_0,\\
  f_j=z_j,\ \  \kappa_0+1\leq j\leq n-1,\\
  \phi_{lk}=\mu_{lk}z_lz_k+\sum_{j=1}^{\kappa_0}z_j\phi^*_{lkj},\ \ (l,k)\in {\cal S}_0,\\
  \phi_{lk}=\sum_{j=1}^{\kappa_0}z_j\phi_{lkj}^*=O_{wt}(3),\ \
 (l, k)\in {\cal S}_1,\\
  g=w,\\
f_{lj}^*(z,w)=\delta_l^j+\frac{i\delta_{l}^j\mu_l}{2}w+b_{lj}^{(1)}(z)w+O_{wt}(4),
\ \ 1\le l\le \kappa_0,\ \mu_l>0,\\
\phi^*_{lkj}(z,w)=O_{wt}(2),\ \  (l,k)\in {\cal S}_1.\ \
  \end{array}\right.
\label{eqn:hao}
\end{equation}
Here, for $1\le \kappa_0\le n-2$, we write ${\cal S} ={\cal S}_0\cup
{\cal S}_1$, the index set  for all components of $\phi$, where
${\cal S}_{0}=\{(j,l): 1\le j\leq \kappa_0, 1\leq l\leq n-1, j\leq
l\}$ and ${\cal S}_1=\{(j, l): j=\kappa_0+1, \kappa_0+1\le l \le
N-n-\frac{(2n-\kappa_0-1)\kappa_0}{2} \}$. Also,
$\mu_{jl}=\sqrt{\mu_j+\mu_l}\ for\ j<l\le \kappa_0$; and $\
\mu_{jl}=\sqrt{\mu_j}$ if $j\le \kappa_0<l$ or if $j=l\le \kappa_0$.
\end{theorem}

Finally, we recall the following lemma of the first author in [Hu1],
which will play a fundamental role in our proof:

\begin{lemma} (Huang, Lemma 3.2 \cite{Hu1})\label{huanglemma}
 Let $k$ be a positive integer such that $1\leq k\leq n-2$. Assume that
$a_1,\cdots, a_k$, $b_1,\cdots, b_k$ are germs at $0\in {\CC}^{n-1}$
of holomorphic functions such that ${a_j(0)=0,\ b_j(0)=0}$ and
\begin{equation}\label{basiceq}
\sum_{i=1}^ka_i(z) \overline{b_i(z)}= A(z,\bar z)|z|^2,\  \
\end{equation}
where $A(z,\bar z)$ is a germ at $0\in {\CC}^{n-1}$ of a real
analytic function. Then $A(z,\bar z)=\sum_{i=1}^ka_i(z)
\overline{b_i(z)}$ $\equiv 0$.
\end{lemma}

\section{ Analysis on the Chern-Moser equation}

Suppose now that
$F=(f,\phi,g)$ is a proper rational map from ${\HH}_n$ into
${\HH}_N$, and satisfies the normalization as in Theorem \ref{thm3}
with $1\le \kappa_0\le n-2$.
Write the codimension part $\phi$ of the map $F$ as
$\phi:=(\Phi_0,\Phi_1)$ with $\Phi_0=(\phi_{\ell k})_{(\ell , k)\in
{\cal S}_0}$ and $\Phi_1=(\phi_{\ell k})_{(\ell, k)\in {\cal S}_1}$.
Write
\[
\Phi_0^{(1,1)}(z)= \sum_{j=1}^{\kappa_0} {e_j} z_j,\ \
\Phi_1^{(1,1)}(z)= \sum_{j=1}^{\kappa_0} \hat{e}_j z_j,
\]
with
{${e_j}$}  $\in {\CC}^{\#({\cal S}_0)}={\CC}^{\kappa_0n-\frac{\kappa_0(\kappa_0+1)}{2}}$,
 $\hat{e}_j \in \CC^{\#({\cal S}_1)}$,
$\xi_j(z) = \overline{{e}_j} \cdot \Phi^{(2,0)}_0(z)$, and $\xi=(\xi_1,
..., \xi_{\kappa_0})$. We also
 write in the following:
\begin{equation*}\label{007}
\begin{split}
&\phi^{(1,1)}(z)w=\sum e_j^*z_jw,\ \hbox{with}\ e_j^*=(e_j,\hat{e}_j),\\
&H=\sum_{(i_1,\cdots,i_{n-1},i_n)}
H^{(i_1,\cdots,i_n)}z_1^{i_1}\cdots
z_{n-1}^{i_{n-1}}w^{i_n}=\sum_{k,j=0}^{\infty} H^{(k,j)}(z)w^j\
\text{for}\ H=f\ \text{or}\ \phi.
\end{split}
\end{equation*}
Here $H^{(k,j)}(z)$ is a homogeneous polynomial of degree $k$ in
$z$.

In this section, we demonstrate our basic idea of the proof through an easier case.
 We proceed with the following lemma, that will be used later:

\begin{lemma}\label{product}
  Let $(\Gamma_{j}^{[h]}(z))_{1\leq j\leq \kappa_0,h=1,2}$ be some
  holomorphic functions of $z$. Let $\mu_{jl}$ and  $\mu_j$ be as in Theorem \ref{thm3}. Suppose that
   for $h=1,2$,  $(\Lambda_{j\ell}^{[h]})_{(j,\ell)\in
  \mathcal{S}_0}$  are defined as follows:
  \begin{enumerate}
\item $\mu_{j \ell} \Lambda^{[h]}_{j \ell}(z) = 2i (z_j \Gamma^{[h]}_\ell + z_\ell
\Gamma^{[h]}_j),\ \ j<\ell \le \kappa_0$,
\item  $\mu_{j j } \Lambda^{[h]}_{j j}(z) = 2i z_j  \Gamma^{[h]}_j (z),\ \ j\le
\kappa_0$,
\item $\mu_{j \ell} \Lambda^{[h]}_{j \ell} = 2i z_\ell \Gamma^{[h]}_j(z),\ \ \ j\le
\kappa_0 < \ell$.
\end{enumerate}
Then we have
\begin{equation}\begin{split}
\sum\limits_{(j,\ell)\in \mathcal{S}_0}\ov{\Lambda^{[1]}_{j
\ell}}\Lambda^{[2]}_{j \ell} =&4|z|^2\Big(\sum_{j\le \kappa_0}
\frac{1}{\mu_j} \ov{\Gamma^{[1]}_j}\Gamma^{[2]}_j \Big)
        -
        \sum_{j < \ell \le \kappa_0}
        \frac{4}{\mu_j \mu_\ell(\mu_j+\mu_\ell)}
        \big(\mu_j\ov{z_j}\ov{\Gamma^{[1]}_\ell}-\mu_\ell\ov{z_\ell}\ov{\Gamma^{[1]}_j}\big)\\
&\cdot \big(\mu_jz_j\Gamma^{[2]}_\ell-\mu_\ell
z_\ell\Gamma^{[2]}_j\big).
\end{split}\end{equation}
\end{lemma}

\begin{proof}
Making use of  the formulas  between $\mu_{j \ell}$ and $\mu_j,\
\mu_\ell$ in Theorem \ref{thm3}, we get, from a straightforward
computation, the following:
\begin{equation*}\begin{split}
\frac{1}{4} \sum\limits_{(j,\ell)\in
\mathcal{S}_0}\ov{\Lambda^{[1]}_{j \ell}}\Lambda^{[2]}_{j \ell}
=&\sum\limits_{1\leq j\leq
\kappa_0}\frac{|z_j|^2}{\mu_j}\ov{\Gamma^{[1]}_{j}}\Gamma^{[2]}_{j
}+\sum\limits_{ j\leq
\kappa_0<\ell}\frac{|z_\ell|^2}{\mu_j}\ov{\Gamma^{[1]}_{j}}\Gamma^{[2]}_{j}\\
&+\sum_{j < \ell \le \kappa_0} \frac{1}{\mu_j+\mu_\ell} \ov{(z_j
\Gamma^{[1]}_\ell + z_\ell \Gamma^{[1]}_j)}\cdot (z_j \Gamma^{[2]}_\ell + z_\ell \Gamma^{[2]}_j)\\
=&\Big( \sum_{j\le \kappa_0} \frac{1}{\mu_j}\ov{\Gamma^{[1]}_{j
}}\Gamma^{[2]}_{j } \Big)\vert z\vert^2 - \underset{\ell \le
\kappa_0, \ell \not=j \le \kappa_0}{\sum} \frac{1}{\mu_j} |z_\ell|^2
\ov{\Gamma^{[1]}_{j}}\Gamma^{[2]}_{j}\\
&+ \sum_{j < \ell \le \kappa_0} \frac{1}{\mu_j+\mu_\ell}
\ov{(z_j\Gamma^{[1]}_{\ell}+z_\ell\Gamma^{[1]}_{j})}\cdot
(z_j\Gamma^{[2]}_{\ell}+z_\ell\Gamma^{[2]}_{j})).
\end{split}\end{equation*}

Now the lemma follows from  the following elementary identity:
\begin{equation*}\begin{split}
& \frac{\mu_j}{\mu_\ell} |z_j
|^2\ov{\Gamma^{[1]}_{\ell}}\Gamma^{[2]}_{\ell} +
\frac{\mu_\ell}{\mu_j} |z_\ell|^2
\ov{\Gamma^{[1]}_{j}}\Gamma^{[2]}_{j} - z_j\Gamma^{[2]}_\ell
\ov{\Gamma^{[1]}_j}\ov{z_\ell} - \ov{z_j}\ov{\Gamma^{[1]}_\ell}
\Gamma^{[2]}_j z_\ell \\& =\frac{1}{\mu_j\mu_\ell}
        \big(\mu_j\ov{z_j}\ov{\Gamma^{[1]}_\ell}-\mu_\ell\ov{z_\ell}\ov{\Gamma^{[1]}_j}\big)
        \cdot \big(\mu_jz_j\Gamma^{[2]}_\ell-\mu_\ell
z_\ell\Gamma^{[2]}_j\big).
\end{split}\end{equation*}  \end{proof}

Next we derive the following formula:

\begin{lemma}\label{30}
\begin{equation}
\label{(2)} \frac{1}{4} \vert \Phi^{(3,0)}_0(z)\vert^2 =
\Big(\sum_{j\le \kappa_0} \frac{1}{\mu_j} |\xi_j(z)|^2 \Big) |z|^2 -
\sum_{j < \ell \le \kappa_0} \frac{1}{\mu_j + \mu_\ell} \Big|
\sqrt{\frac{\mu_j}{\mu_\ell}} z_j \xi_\ell -
\sqrt{\frac{\mu_\ell}{\mu_j}} z_\ell \xi_j\Big|^2.
\end{equation}
\end{lemma}

\begin{proof} Since
\begin{equation}
\label{*}
-\text{Im}\big(g(z,w)\big)+\big|f(z,w)\big|^2+\big|\phi(z,w)\big|^2=0\
\ \ \hbox{over}\ \ \ \text{Im}(w)=|z|^2,
\end{equation}
we can consider terms of weighted degree $5$ to get,  over
Im$(w)=|z|^2$, the following
\begin{equation*}
z \ov{f^{(4)}(z,w)}+ \ov z f^{(4)}(z,w) + \Phi^{(2)}_0(z,w)
\ov{\Phi^{(3)}_0(z,w)} + \Phi^{(3)}_0(z,w)
\ov{\Phi^{(2)}_0(z,w)}=0,\  \hbox{or}
\end{equation*}
\begin{equation}\begin{split}
&z \ov{f^{(2,1)}(z)(u+i|z|^2)} + \ov z f^{(2,1)}(z)(u+i |z|^2) +
\Phi^{(2)}_0(z) \ov{\bigg(\Phi^{(3,0)}_0(z)+(\sum e_j z_j)w\bigg)}\\
&+ \bigg( \Phi^{(3,0)}_0(z)+(\sum e_j z_j) w\bigg)
\ov{\Phi^{(2)}_0(z)}\equiv 0.
\end{split}\end{equation}
Here, we know $f^{(4)}(z,w)=f^{(2,1)}(z)w$ by the above mentioned
normalization. Collecting terms of the form $\ov z^\alpha z^\beta u$
with $|\alpha|=1$, $|\beta|=2$, we get
\[
\ov z f^{(2,1)}(z) + \Phi^{(2)}_0(z) \ov{\sum e_j z_j} = 0, \ \ or,
\ \
\]
\begin{equation}\label{hh}
\ov z f^{(2,1)}(z) = - \ov{\text{{$(z_1, ..., z_{\kappa_0})$}}}\cdot
\xi(z).
\end{equation}
Collecting terms of the form $z^\alpha \ov z^\beta$ with
$|\alpha|=3$ and $|\beta|=2$, we get
\begin{equation}
i \ov z f^{(2,1)}(z) |z|^2 +\ov{\Phi^{(2)}_0(z)}{ \Phi^{(3,0)}_0(z)}
+ {\Phi^{(2)}_0(z)} \ov{\sum^{\kappa_0}_{j=1} e_j z_j  (i |z|^2)}
\equiv 0.
\end{equation}
We thus get
\begin{equation}
\label{(3)'} \ov{\Phi^{(2)}_0 (z)}{ \Phi^{(3,0)}_0(z)}=2 i \ov
{(z_1,\cdots,z_{\kappa_0})}\cdot \xi(z) |z|^2.
\end{equation}
Equivalently,  we have
\begin{enumerate}
\item $\mu_{j \ell} \phi^{(3, 0)}_{j \ell}(z) = 2i (z_j \xi_\ell(z) + z_\ell
\xi_j(z)),\ \ j<\ell \le \kappa_0$,
\item  $\mu_{j j } \phi^{(3, 0)}_{j j}(z) = 2i z_j  \xi_j (z),\ \ j\le
\kappa_0$,
\item $\mu_{j \ell} \phi^{(3, 0)}_{j \ell}(z) = 2i z_\ell \xi_j(z),\ \ \ j\le
\kappa_0 < \ell$.
\end{enumerate}
Now Lemma \ref{30} follows from Lemma \ref{product}.\end{proof}

\medskip
\begin{lemma}
\label{lemma 2} $\vert \phi^{(3,0)}(z) \vert^2$ $  = A(z, \ov
z)\vert z\vert^2$ with $A(z, \ov z)$ a real analytic polynomial  in
$(z,\ov{z})$.
\end{lemma}

\noindent{\it Proof:}\ \ \ \ \ Collecting terms of weighted degree
$6$ in (\ref{*}), we get
\begin{equation}\begin{split}
&z \ov{f^{(5)}(z,w)} + \ov z f^{(5)}(z,w) + \Phi^{(2)}_0(z,w)\cdot
\ov{\Phi^{(4)}_0(z,w)} + \ov{\Phi^{(2)}_0(z,w)}\cdot \Phi^{(4)}_0(z,w)\\
& + | \phi^{(3)}(z,w)|^2 + |f^{(3)}(z,w)|^2=0\ \ \hbox{over}\
\text{Im}(w)=|z|^2.
\end{split}\end{equation}
Collecting terms of the form $z^\alpha \ov
z^\beta$ with $|\alpha|=|\beta|=3$ and applying the normalization
for $F$, we easily see  the proof (cf., (\ref{330}) below) . \ \
$\Box$

\medskip

Notice that $\vert \phi^{(3,0)}(z) \vert^2$ $=\vert
\Phi_0^{(3,0)}(z) \vert^2 +\vert \Phi_1^{(3,0)}(z) \vert^2$ and
there are $\frac{\kappa_0(\kappa_0+1)}{2}-\kappa_0$ negative terms
in the right hand side of (\ref{(2)}). Also there are
$(N-(\kappa_0+1)n+\frac{\kappa_0(\kappa_0+1)}{2})$ components in
$\Phi_1$. Applying  Lemma \ref{huanglemma} and [Proposition 3 on
page 102 of \cite{DA}], we immediately get, after applying a unitary
transformation to the $\Phi_1$-components, the following:

\medskip
\begin{Corollary}\label{hjy1}
\label{lemma 3} Suppose that $\kappa_0\ge 2$ and
{$(\kappa_0+1)n-\kappa_0\le$} $N\le (\kappa_0+2)n -
\kappa_0(\kappa_0+1)+\kappa_0-2$. Then
\begin{equation}\begin{split}
&\Phi^{(3,0)}_1(z)=\Big
(\frac{2}{\sqrt{\mu_j+\mu_l}}\big(\sqrt{\frac{\mu_j}{\mu_\ell}} z_j
\xi_\ell -\sqrt{ \frac{\mu_\ell}{\mu_j}} z_\ell \xi_j\big),
0'\Big)_{1\le j  < l \le \kappa_0}, \\
&  \vert \phi^{(3,0)}(z)\vert^2 = 4\Big(\sum_{j\le \kappa_0}
\frac{1}{\mu_j} |\xi_j(z)|^2 \Big) |z|^2.
\end{split}\end{equation}
\end{Corollary}

Here we notice that the condition $N\ge (\kappa_0+1)n-\kappa_0$
implies that $N>\#({\cal S}_0)+n$ for $\kappa_0>1$.

\section{ Partial linearity and further applications
of the Chern-Moser equation}
We  assume in this section that
the proper rational  map $F=(f,\phi,g)$ from ${\HH}_n$ into
${\HH}_{N}$ satisfies the normalization as in Theorem \ref{thm3}
with $\kappa_0=2$. Moreover, by what is done in the last section, we
assume that $\Phi^{(3,0)}_1(z)$ has been normalized to take the form
as in Corollary \ref{hjy1}. Namely, the only possible non-zero
element in $\Phi^{(3,0)}_1(z) $ is $\phi_{33}^{(3,0)}(z) $.

In this section, we prove the following result, which will be
crucial for our proof of Theorem \ref{thm1}:

\begin{theorem} \label{prop11} Assume that $F$ is as in Theorem \ref{thm3} with
$\kappa_0=2$, {$n\ge 7$} and {$3n-2\le $} $N\le 4n-6$. Also, assume
that $\Phi_1^{(3,0)}(z)$ is normalized as in Corollary \ref{hjy1}.
Then the following holds:

\noindent (1):
$\Phi_1^{(4,0)}(z)=(\phi_{33}^{(4,0)}(z),0,\cdots,0)$, where
\[
\phi^{(4,0)}_{33}(z)=\frac{ 2 }{\sqrt{\mu_1+\mu_2}}
\bigg(\sqrt{\frac{\mu_1}{\mu_2}}z_1 \eta_2^*  -
\sqrt{\frac{\mu_2}{\mu_1}} z_2 \eta_1^* \bigg),\ \ \eta_1^* =
\phi^{(3,0)}(z)\cdot \ov{e_1^*},\ \ \eta_2^* = \phi^{(3,0)}(z)\cdot
\ov{e_2^*}.
\]

\noindent (2): $D^{\alpha}_z\Phi_1^{(2,1)}(z)\in \hbox{span}
\{(1,0,\cdots,0), \hat{e}_1,\hat{e}_2\}$ for $|\alpha|=2$.

\noindent (3): $D^{\alpha}_z\Phi_1^{(1,2)}(z)\in \hbox{span}
\{\hat{e}_1,\hat{e}_2\}$ for $|\alpha|=1$.

\noindent Here $\hat{e}_1, \hat{e}_2, e_1^*, e_2^*,\ $ are  defined
as at the beginning of the last section, and { $D$ is the regular
differential operator.}
\end{theorem}

This section is devoted to the proof of  Theorem \ref{prop11}.

\medskip
Notice that $g=w$.  By the partial linearity theorem of the first
author proved  in [Hu2], we can assume that for any
$\epsilon=(\epsilon_1,\epsilon_2)(\in \mathbb{C}^2)\approx 0$, there
is a unique affine subspace $L_\epsilon$ of codimension two  defined
by equations of the form:
\begin{equation}
  z_1=\sum\limits_{i=3}^{n-1}a_i(\epsilon)z_i+a_n(\epsilon)w+\epsilon_1,\
  z_2=\sum\limits_{i=3}^{n-1}b_i(\epsilon)z_i+b_n(\epsilon)w+\epsilon_2,\ a_i(0)=b_i(0)=0
\end{equation}
such that $F$ is a linear map on $L_{\epsilon}$. Here $a_j,b_j$ are
holomorphic functions in $\epsilon$ near $0$. Hence we have
$$
\frac{\p^2 H}{\p w^2}\Big|_{L_\epsilon}=0\ \text{for}\ H=f\
\text{or}\ \phi.
$$
Namely, for {$H(L_{\epsilon})=H\big(\sum^{n-1}_{i=3}a_i(\epsilon)z_i
+ a_n(\epsilon)w+\epsilon_1, \ \sum^{n-1}_{i=3}b_i(\epsilon)z_i +
b_n(\epsilon)w+\epsilon_2, z_3, ..., z_{n-1}, w\big)$}, we have {
\begin{equation}\begin{split}\label{w2}
 & 0=\frac{\p^2 H(L_{\epsilon})}{\p w^2}\Big|_{(\epsilon_1,\epsilon_2)}\\
 & =\Big(\frac{\p^2 H}{\p z_1^{2}}a_n^2+\frac{\p^2 H}{\p z_2^{2}}b_n^2
 + 2\frac{\p^2 H}{\p z_1 \p z_2}a_nb_n+2\frac{\p^2 H}{\p z_1\p w}a_n
 + 2\frac{\p^2 H}{\p z_2 \p w}b_n+\frac{\p^2 H}{\p w^{2}}\Big)
 \Big|_{(\epsilon_1,\epsilon_2, 0 ,..., 0)}.
 \end{split}
\end{equation}
}
Let $a_n^{(1)}(\epsilon)$ and $b_n^{(1)}(\epsilon)$ be the linear
parts in $a_n$ and $b_n$, respectively. Set $H=f_1,f_2$ and $\phi$
in (\ref{w2}), respectively. We then get
\begin{equation}\begin{split}
&\frac{i}{2}\mu_1a_n^{(1)}(\epsilon)+ f_1^{(1,2)}(\epsilon\text{{$,
0, ..., 0$}})=0,\
\frac{i}{2}\mu_2b_n^{(1)}(\epsilon)+ f_2^{(1,2)}(\epsilon\text{{$, 0, ..., 0$}})=0,\\
&\phi^{(1,2)}(\epsilon\text{{$, 0, ...,
0$}})+e^*_1a_n^{(1)}(\epsilon)+e^*_2b^{(1)}_n(\epsilon)=0.
\end{split}\end{equation}
Notice that by Theorem \ref{thm3}, $F^{(1,m)}(z)$ depends only on
$(z_1,z_2)$ for any $m$. It then follows:
\begin{equation}\begin{split}\label {eq1}
\phi^{(1,2)}(\epsilon\text{{$, 0, ...,
0$}})=-e_1^*a_n^{(1)}(\epsilon)-e_2^*b^{(1)}_n(\epsilon)
=-\frac{2i}{\mu_1}f_1^{(1,2)}(\epsilon\text{{$, 0, ...,
0$}})e^*_1-\frac{2i}{\mu_2}f_2^{(1,2)}(\epsilon\text{{$, 0, ...,
0$}})e_2^*.
\end{split}\end{equation}
This proves Theorem \ref{prop11} (3).  Moreover,  we obtain
\begin{equation}\begin{split}\label{eq01}
\ov{\Phi_0^{(1,2)}(z)}\cdot
\Phi_0^{(2,0)}(z)&=\frac{2i}{\mu_1}\ov{f_1^{(1,2)}(z)}\ov{e_1}\cdot
\Phi_0^{(2,0)}(z)
     +\frac{2i}{\mu_2}\ov{f_2^{(1,2)}(z)}\ov{e_2}\cdot
\Phi_0^{(2,0)}(z)\\
&=\frac{2i}{\mu_1}(\ov{f_1^{(I_1+2I_n)}}\ov{z_1}+\ov{f_1^{(I_2+2I_n)}}\ov{z_2})\xi_1
     +\frac{2i}{\mu_2}(\ov{f_2^{(I_1+2I_n)}}\ov{z_1}+\ov{f_2^{(I_2+2I_n)}}\ov{z_2})\xi_2.
\end{split}\end{equation}
{Here and in} what follows, write $I_j=(0,\cdots, 0,
1,0,\cdots,0)\in {\ZZ}^n$, where the non-zero element $1$ is in the
$j^{th}$-position. From (\ref{eq01}), we also have
\begin{equation}\begin{split}\label {eq2}
\ov{\Phi_0^{(I_1+2I_n)}}\cdot
\Phi_0^{(2,0)}(z)&=2i(\frac{\xi_1}{\mu_1}\ov{f_1^{(I_1+2I_n)}}+\frac{\xi_2}{\mu_2}\ov{f_2^{(I_1+2I_n)}})\\
\ov{\Phi_0^{(I_2+2I_n)}}\cdot
\Phi_0^{(2,0)}(z)&=2i(\frac{\xi_1}{\mu_1}\ov{f_1^{(I_2+2I_n)}}+\frac{\xi_2}{\mu_2}\ov{f_2^{(I_2+2I_n)}}),
\end{split}\end{equation}
and the following:
\begin{equation}\begin{split}\label{97eq1}
&2i\big(\frac{\ov{\xi_1}}{\mu_1}\ov{\Phi_0^{(I_1+2I_n)}}\cdot
\Phi_0^{(2,0)}(z)+\frac{\ov{\xi_2}}{\mu_2}\ov{\Phi_0^{(I_2+2I_n)}}\cdot\Phi_0^{(2,0)}(z)\big)\\
&=\frac{-4\ov{\xi_1}}{\mu_1}\cdot\big(\frac{\xi_1}{\mu_1}\ov{f_1^{(I_1+2I_n)}}
+\frac{\xi_2}{\mu_2}\ov{f_2^{(I_1+2I_n)}}\big)
+\frac{-4\ov{\xi_2}}{\mu_2}\big(\frac{\xi_1}{\mu_1}\ov{f_1^{(I_2+2I_n)}}+\frac{\xi_2}{\mu_2}\ov{f_2^{(I_2+2I_n)}}\big)\\
&=-4\frac{\xi_1}{\mu_1}
   \big(\ov{f_{1}^{(I_1+2I_n)}}\frac{\ov{\xi_1}}{\mu_1}+\ov{f_{1}^{(I_2+2I_n)}}\frac{\ov{\xi_2}}{\mu_2}\big)
   -4\frac{\xi_2}{\mu_2}
   \big(\ov{f_{2}^{(I_1+2I_n)}}\frac{\ov{\xi_1}}{\mu_1}+\ov{f_{2}^{(I_2+2I_n)}}\frac{\ov{\xi}_2}{\mu_2}\big).
\end{split}\end{equation}

Considering terms of weighted degree $6$  in the basic equation
(\ref{*}), we get
\begin{equation}\begin{split}
2\text{Re}\Big\{\ov{z}f^{(5)}(z,w)+\ov{\Phi_0^{(2)}(z,w)}\cdot{\Phi_0^{(4)}(z,w)}\Big\}
+\big|f^{(3)}(z,w)\big|^2+\big|\phi^{(3)}(z,w)\big|^2=0
\end{split}\end{equation}
over Im$(w)=|z|^2$.  Namely, we have
\begin{equation}\begin{split}
&2\text{Re}\Big\{\ov{z}\Big(f^{(3,1)}(z)(u+i|z|^2)+f^{(1,2)}(z)(u+i|z|^2)^2\Big)
+\ov{\Phi_0^{(2,0)}(z)}\Big(\Phi_0^{(4,0)}(z)
   +\Phi_0^{(2,1)}(z)\\
   &\ \ \cdot(u+i|z|^2)\Big)\Big\}
+\big|f^{(1,1)}(z)(u+i|z|^2)\big|^2+\big|\phi^{(3,0)}(z)+\phi^{(1,1)}(z)(u+i|z|^2)\big|^2=0.
\end{split}\end{equation}
{Here we notice that the $f^{(5,0)}(z)$ term is not involved (cf.
\cite{HJX1}, Lemma 2.3(A)).} Collecting terms of  the form
$z^{\a}\ov{z^{\b}}u^2$ with $|\a|=1,|\b|=1$, we get
\begin{equation}\begin{split}\label{112}
&2\text{Re}\big(\ov{z}f^{(1,2)}(z)\big)+|f^{(1,1)}(z)|^2+|\phi^{(1,1)}(z)|^2=0.
\end{split}\end{equation}
Collecting terms of  the form $z^{\a}\ov{z}^{\b}u$ with
$|\a|=3,|\b|=1$, we get
\begin{equation}\begin{split}\label{311}
&\ov{z}f^{(3,1)}(z)+\phi^{(3,0)}(z)\cdot\ov{\phi^{(1,1)}(z)}=0.
\end{split}\end{equation}
Collecting terms of  the form $z^{\a}\ov{z}^{\b}u$ with
$|\a|=2,|\b|=2$, we get
\begin{equation}\begin{split}\label{221}
&2\text{Re}\Big(2i\ov{z}f^{(1,2)}(z)|z|^2+\ov{\Phi_0^{(2,0)}(z)}\cdot
{\Phi_0^{(2,1)}(z)}\Big)=0.
\end{split}\end{equation}
Collecting terms of the form  $z^{\a}\ov{z}^{\b}$ with
$|\a|=4,|\b|=2$, we get
\begin{equation}\begin{split}\label{420}
&i|z|^2\ov{z}f^{(3,1)}(z)+\ov{\Phi_0^{(2,0)}(z)}\cdot{\Phi_0^{(4,0)}(z)}-i|z|^2\ov{\phi^{(1,1)}(z)}\cdot
{\phi^{(3,0)}(z)}=0.
\end{split}\end{equation}
Collecting terms of the form  $z^{\a}\ov{z}^{\b}$ with
$|\a|=3,|\b|=3$, we get
\begin{equation}\begin{split}\label{330}
&2\text{Re}\Big(-\ov{z}f^{(1,2)}(z)|z|^4+i|z|^2\ov{\Phi_0^{(2,0)}(z)}\cdot
{\Phi_0^{(2,1)}(z)}\Big)\\
&+|z|^4\cdot|f^{(1,1)}(z)|^2+|\phi^{(3,0)}(z)|^2+|z|^4\cdot|\phi^{(1,1)}(z)|^2=0.
\end{split}\end{equation}
Combining (\ref{311}) with (\ref{420}), we get
\begin{equation}\begin{split}\label{92eq1}
\ov{\Phi_0^{(2,0)}(z)}\cdot{\Phi_0^{(4,0)}(z)}=2i|z|^2\ov{\phi^{(1,1)}(z)}\cdot
{\phi^{(3,0)}(z)}.
\end{split}\end{equation}
Substituting (\ref{112}) into (\ref{330}), we get
\begin{equation}\begin{split}\label{92eq2}
&2\text{Re}\Big(-2\ov{z}f^{(1,2)}(z)|z|^2+i\ov{\Phi_0^{(2,0)}(z)}\cdot
{\Phi_0^{(2,1)}(z)}\Big)|z|^2+|\phi^{(3,0)}(z)|^2=0.
\end{split}\end{equation}
Combining (\ref{221}) with (\ref{92eq2}), we get
\begin{equation}\begin{split}\label{92eq3}
&2\Big(-2\ov{z}f^{(1,2)}(z)|z|^2+i\ov{\Phi_0^{(2,0)}(z)}\cdot
{\Phi_0^{(2,1)}(z)}\Big)|z|^2+|\phi^{(3,0)}(z)|^2=0.
\end{split}\end{equation}
Recall that in Corollary \ref{hjy1}, we have obtained
\begin{equation}\begin{split}
|\phi^{(3,0)}(z)|^2=4|z|^2\Big(\frac{|\xi_1|^2}{\mu_1}+\frac{|\xi_2|^2}{\mu_2}\Big).
\end{split}\end{equation}
 Hence we have
\begin{equation}\begin{split}
2\Big(-2\ov{z}f^{(1,2)}(z)|z|^2+i\ov{\Phi_0^{(2,0)}(z)}\cdot
{\Phi_0^{(2,1)}(z)}\Big)+4\Big(\frac{|\xi_1|^2}{\mu_1}+\frac{|\xi_2|^2}{\mu_2}\Big)=0.
\end{split}\end{equation}
Notice that $|\xi_i|^2=\ov{\xi_i}\xi_i=\xi_ie_i\cdot
\ov{\Phi_0^{(2,0)}(z)}$. Set
\begin{equation}\label{2013 4.20}
\begin{split}
\w{\phi}^{(2,1)}(z)=\phi^{(2,1)}(z)-2i\sum\limits_{j=1}^{2}\frac{\xi_j}{\mu_j}e^*_j,
\ \ \text{{$\w \Phi^{(2,1)}_0(z)=\Phi^{(2,1)}_0(z) -2i \sum^2_{j=1}
\frac{\xi_j}{\mu_j}e_j $}}.
\end{split}\end{equation}
Then we have
\begin{equation}\begin{split}
\ov{\Phi_0^{(2,0)}(z)}\cdot\w{\Phi}_0^{(2,1)}(z)=-2i|z|^2\ov{z}\cdot
f^{(1,2)}(z).
\end{split}\end{equation}
Recall $f_j=z_j$ for $3\le j\le n-1$ so that $\ov {z} \cdot
f^{(1,2)}(z)=(\ov {z_1}, \ov {z_2})\cdot (f^{(1,2)}_1(z),
f^{(1,2)}_2(z))$ and thus we get
\begin{equation}\begin{split}
&\w{\Phi}_{11}^{(2,1)}(z)=\frac{-2i}{\sqrt{\mu_1}}z_1f_1^{(1,2)}(z),\
\w{\Phi}_{12}^{(2,1)}(z)=\frac{-2i}{\sqrt{\mu_1+\mu_2}}\big(z_1f_2^{(1,2)}(z)+z_2f_1^{(1,2)}(z)\big),\\
&\w{\Phi}_{22}^{(2,1)}(z)=\frac{-2i}{\sqrt{\mu_2}}z_2f_2^{(1,2)}(z),\\
&\w{\Phi}_{1j}^{(2,1)}(z)=\frac{-2i}{\sqrt{\mu_1}}z_jf_1^{(1,2)}(z),\
\w{\Phi}_{2j}^{(2,1)}(z)=\frac{-2i}{\sqrt{\mu_2}}z_jf_2^{(1,2)}(z)\
\text{with}\ j\geq 3.
\end{split}\end{equation}
Making use of Lemma \ref{product}, we get
\begin{equation}\begin{split}\label{97eq8}
|\w{\Phi}_{0}^{(2,1)}(z)|^2=4|z|^2\sum\limits_{j=1}^{2}\frac{1}{\mu_j}|f_j^{(1,2)}(z)|^2
-\frac{4}{\mu_1\mu_2(\mu_1+\mu_2)}\big|\mu_1z_1f_2^{(1,2)}(z)-\mu_2z_2f_1^{(1,2)}(z)\big|^2,
\end{split}\end{equation}
and
\begin{equation}\begin{split}\label{97eq11}
\ov{\w{\Phi}_{0}^{(2,1)}(z)}\Phi_0^{(3,0)}\text{{$(z)$}}=&-4|z|^2\sum\limits_{j=1}^{2}\frac{1}{\mu_j}\ov{f_j^{(1,2)}(z)}\xi_j\\
&+\frac{4}{\mu_1\mu_2(\mu_1+\mu_2)}\ov{\Big(\mu_1z_1f_2^{(1,2)}(z)-\mu_2z_2f_1^{(1,2)}(z)\Big)}
\cdot \Big(\mu_1z_1\xi_2-\mu_2z_2\xi_1\Big).
\end{split}\end{equation}

Notice that if we replace $z_1,z_2$ by $\frac{\xi_1}{\mu_1}$,
\text{{$\frac{\xi_2}{\mu_2}$}}, respectively, in (\ref{112}), we get
\begin{equation}\begin{split}
&2\text{Re}\Big\{\frac{\ov{\xi_1}}{\mu_1}
   \Big(f_{1}^{(I_1+2I_n)}\frac{\xi_1}{\mu_1}+f_{1}^{(I_2+2I_n)}\frac{\xi_2}{\mu_2}\Big)
   +\frac{\ov{\xi_2}}{\mu_2}
   \Big(f_{2}^{(I_1+2I_n)}\frac{\xi_1}{\mu_1}+f_{2}^{(I_2+2I_n)}\frac{\xi_2}{\mu_2}\Big)\Big\}\\
&+\frac{1}{4}(|\xi_1|^2+|\xi_2|^2)+|\frac{\xi_1}{\mu_1}e^*_1+\frac{\xi_2}{\mu_2}e^*_2|^2=0.
\end{split}\end{equation}
Here we have used \text{{$f_j^{(1,1)}(z)=\frac{i}{2}\mu_j z_j$}} for
\text{{$j=1,2.$}} Combining this with (\ref{97eq1}), we get
\begin{equation}\begin{split}\label{97eq7}
&-2\text{Re}\Big\{2i\Big(\frac{\ov\xi_1}{\mu_1}\ov{\Phi_0^{(I_1+2I_n)}}
+\frac{\ov\xi_2}{\mu_2}\ov{\Phi_0^{(I_2+2I_n)}}\Big)\cdot\Phi_0^{(2,0)}(z)\Big\}\\
&+(|\xi_1|^2+|\xi_2|^2)+4\big|\frac{\xi_1}{\mu_1}e^*_1+\frac{\xi_2}{\mu_2}e^*_2\big|^2=0.
\end{split}\end{equation}

\medskip


Considering terms of weighted degree $7$ in the basic equation
(\ref{*}), we get
\begin{equation}\begin{split}
2\text{Re}\Big\{&\ov{z}f^{(6)}(z,w)+\ov{f^{(3)}(z,w)}
f^{(4)}(z,w)+\ov{\Phi_0^{(2)}(z,w)}{\Phi_0^{(5)}(z,w)}\\
&+\ov{\phi^{(3)}(z,w)}{\phi^{(4)}}(z,w)\Big\}=0
\end{split}\end{equation}
over Im$(w)=|z|^2$. Namely, we have
\begin{equation}\begin{split}
&2\text{Re}\Big\{\ov{z}\Big(f^{(4,1)}(z)(u+i|z|^2)+f^{(2,2)}(z)(u+i|z|^2)^2\Big)+
\ov{f^{(1,1)}(z)(u+i|z|^2)}\cdot
f^{(2,1)}(z)\\
&\cdot(u+i|z|^2)+\ov{\Phi_0^{(2,0)}(z)}\Big(\Phi_0^{(5,0)}(z)
   +\Phi_0^{(3,1)}(z)(u+i|z|^2)+\Phi_0^{(1,2)}(z)(u+i|z|^2)^2\Big)\\
&+\ov{\big(\phi^{(3,0)}(z)+\phi^{(1,1)}(z)(u+i|z|^2)\big)}\cdot
\big(\phi^{(4,0)}(z)
   +\phi^{(2,1)}(z)(u+i|z|^2)\big)   \Big\}=0.
\end{split}\end{equation}
{Here we notice that the $f^{(6,0)}(z)$ term is not involved (cf.
\cite{HJX1}, Lemma 2.3(A)).} Collecting terms of the form
$z^{\a}\ov{z}^{\b}u^2$ with $|\a|=2,|\b|=1$, we get
\begin{equation}\begin{split}\label{212}
&\ov{z}f^{(2,2)}(z)+\ov{f^{(1,1)}(z)}\cdot f^{(2,1)}(z)
+\ov{\Phi_0^{(1,2)}(z)}\cdot
\Phi_0^{(2,0)}(z)+\ov{\phi^{(1,1)}(z)}\cdot \phi^{(2,1)}(z)=0.
\end{split}\end{equation}
Collecting terms of  the form $z^{\a}\ov{z}^{\b}u$ with
$|\a|=3,|\b|=2$, we get
\begin{equation}\begin{split}\label{321}
2i\ov{z}|z|^2f^{(2,2)}(z)+\ov{\Phi_0^{(2,0)}(z)}\cdot
\Phi_0^{(3,1)}(z)-2i|z|^2\ov{\Phi_0^{(1,2)}(z)}\cdot
\Phi_0^{(2,0)}(z)+\ov{\phi^{(2,1)}(z)}\cdot \phi^{(3,0)}(z)=0.
\end{split}\end{equation}
Collecting terms of  the form $z^{\a}\ov{z}^{\b}$ with
$|\a|=4,|\b|=3$, we get
\begin{equation}\begin{split}\label{430}
&-\ov{z}f^{(2,2)}(z)|z|^4+\ov{f^{(1,1)}(z)}\cdot
f^{(2,1)}(z)|z|^4+i|z|^2\ov{\Phi_0^{(2,0)}(z)}\cdot\Phi_0^{(3,1)}(z)-|z|^4\ov{\Phi_0^{(1,2)}(z)}\cdot
\Phi_0^{(2,0)}(z)\\
&+\ov{\phi^{(3,0)}(z)}\cdot
\phi^{(4,0)}(z)-i|z|^2\ov{\phi^{(2,1)}(z)}\cdot
\phi^{(3,0)}(z)+|z|^4\ov{\phi^{(1,1)}(z)}\cdot \phi^{(2,1)}(z)=0.
\end{split}\end{equation}

By calculating (\ref{430})$-$ $|z|^4\cdot$(\ref{212}), we get
\begin{equation}\begin{split}\label{92eq4}
&-2\ov{z}f^{(2,2)}(z)|z|^4+i|z|^2\ov{\Phi_0^{(2,0)}(z)}\Phi_0^{(3,1)}(z)-2|z|^4\ov{\Phi_0^{(1,2)}(z)}\cdot
\Phi_0^{(2,0)}(z)\\
&+\ov{\phi^{(3,0)}(z)}\cdot
\phi^{(4,0)}(z)-i|z|^2\ov{\phi^{(2,1)}(z)}\cdot \phi^{(3,0)}(z)=0.
\end{split}\end{equation}
By calculating (\ref{92eq4})$-i|z|^2$(\ref{321}), we get
\begin{equation}\begin{split}\label{92eq5}
\ov{\phi^{(3,0)}(z)}\cdot
\phi^{(4,0)}(z)=4|z|^4\ov{\Phi_0^{(1,2)}(z)}\cdot
\Phi_0^{(2,0)}(z)+2i|z|^2\ov{\phi^{(2,1)}(z)}\cdot \phi^{(3,0)}(z).
\end{split}\end{equation}
Combining this with (\ref{321}), we get
\begin{equation}\begin{split}\label{92eq05}
\ov{\phi^{(3,0)}(z)}\cdot
\phi^{(4,0)}(z)=-2i|z|^2\Big(2i|z|^2\ov{z}f^{(2,2)}(z)+\ov{\Phi_0^{(2,0)}(z)}\cdot\Phi_0^{(3,1)}(z)\Big).
\end{split}\end{equation}

By (\ref{92eq1}), we have
\begin{equation}\begin{split}\label{915eq1}
&\mu_{11}\cdot \Phi_{11}^{(4,0)}(z)=2iz_1\phi^{(3,0)}(z)\cdot \ov{e^*_1},\\
&\mu_{12}\cdot \Phi_{12}^{(4,0)}(z)=2iz_1\phi^{(3,0)}(z)\cdot \ov{e^*_2}+2iz_2\phi^{(3,0)}(z)\cdot \ov{e^*_1},\\
&\mu_{22}\cdot \Phi_{22}^{(4,0)}(z)=2iz_2\phi^{(3,0)}(z)\cdot \ov{e^*_2},\\
&\mu_{1j}\cdot \Phi_{1j}^{(4,0)}(z)=2iz_j\phi^{(3,0)}(z)\cdot \ov{e^*_1},\  j\ge 3,\\
&\mu_{2j}\cdot \Phi_{2j}^{(4,0)}(z)=2iz_j\phi^{(3,0)}(z)\cdot
\ov{e^*_2},\  j\ge 3.
\end{split}\end{equation}
Write
$$
\eta^*_1=\phi^{(3,0)}(z)\cdot \ov{e^*_1},\
\eta^*_2=\phi^{(3,0)}(z)\cdot \ov{e^*_2};\
\eta_1=\Phi_0^{(3,0)}(z)\cdot \ov{e_1},\
\eta_2=\Phi_0^{(3,0)}(z)\cdot \ov{e_2}.
$$
Making use of Lemma \ref{product}, we get
\begin{equation}\begin{split}\label{97eq10}
\ov{\Phi_0^{(3,0)}(z)}\Phi_0^{(4,0)}(z)
       =&4|z|^2\Big(\frac{\ov{\xi_1}\eta^*_1}{\mu_1}+\frac{\ov{\xi_2}\eta^*_2}{\mu_2}\Big)-\frac{4}{\mu_1+\mu_2}
        \Big(\sqrt{\frac{\mu_2}{\mu_1}}\ov{z_2}\ov{\xi_1}-\sqrt{\frac{\mu_1}{\mu_2}}\ov{z_1}\ov{\xi_2}\Big)\\&
        \cdot
         \Big(\sqrt{\frac{\mu_2}{\mu_1}}z_2\eta^*_1-\sqrt{\frac{\mu_1}{\mu_2}}z_1\eta^*_2\Big).
\end{split}\end{equation}
Combining (\ref{92eq05}) with (\ref{97eq10}) and making use of Lemma
\ref{huanglemma}, we get
\[
\ov{\Phi_{1}^{(3,0)}(z)}\Phi_{1}^{(4,0)}(z)=\frac{4}{\mu_1+\mu_2}
        \big(\sqrt{\frac{\mu_2}{\mu_1}}\ov{z_2}\ov{\xi_1}-\sqrt{\frac{\mu_1}{\mu_2}}\ov{z_1}\ov{\xi_2}\big)\cdot
         \big(\sqrt{\frac{\mu_2}{\mu_1}}z_2\eta^*_1-\sqrt{\frac{\mu_1}{\mu_2}}z_1\eta^*_2\big).
  \]
Now, by Corollary \ref{hjy1}, we have
\begin{equation}\begin{split}\label{99eq1}
\phi_{33}^{(4,0)}(z)=
\frac{2}{\sqrt{\mu_1+\mu_2}}\big(\sqrt{\frac{\mu_1}{\mu_2}}z_1\eta^*_2-\sqrt{\frac{\mu_2}{\mu_1}}z_2\eta^*_1\big).
\end{split}\end{equation}
Moreoever
\begin{equation}\begin{split}
2i(\frac{\ov{\xi_1}\eta^*_1}{\mu_1}+\frac{\ov{\xi_2}\eta^*_2}{\mu_2})
=2i|z|^2\ov{z}f^{(2,2)}(z)+\ov{\Phi_0^{(2,0)}(z)}\cdot\Phi_0^{(3,1)}(z).
\end{split}\end{equation}
Write
\begin{equation}\begin{split}
\w{\phi}^{(3,1)}(z)=\phi^{(3,1)}(z)-2i(\frac{\eta^*_1}{\mu_1}e^*_1+\frac{\eta^*_2}{\mu_2}e^*_2),\
\ \text{{$\w \Phi^{(3,1)}_0(z)=\Phi^{(3,1)}_0(z)
-2i(\frac{\eta_1^*}{\mu_1}e_1 + \frac{\eta_2^*}{\mu_2} e_2)$}}.
\end{split}\end{equation}
Then we have
$$
\ov{\Phi_0^{(2,0)}(z)}\wt{\Phi_0}^{(3,1)}(z)=-2i|z|^2\ov{z}f^{(2,2)}(z).$$
Hence, we get

\begin{equation}\begin{split}
&\mu_{11}\cdot \w{\Phi}_{11}^{(3,1)}(z)=-2iz_1f_1^{(2,2)}(z),\\
&\mu_{12}\cdot \w{\Phi}_{12}^{(3,1)}(z)=-2i\big(z_1f_2^{(2,2)}(z)+z_2f_1^{(2,2)}(z)\big),\\
&\mu_{22}\cdot \w{\Phi}_{22}^{(3,1)}(z)=-2iz_2f_2^{(2,2)}(z),\\
&\mu_{1j}\cdot \w{\Phi}_{1j}^{(3,1)}(z)=-2iz_jf_1^{(2,2)}(z),\ j\ge 3,\\
&\mu_{2j}\cdot \w{\Phi}_{2j}^{(3,1)}(z)=-2iz_jf_2^{(2,2)}(z),\ j\ge
3.
\end{split}\end{equation}
By Lemma \ref{product}, we have
\begin{equation}\begin{split}
\ov{{\Phi}_0^{(3,0)}(z)}\w{\Phi}_0^{(3,1)}(z)=&-4|z|^2\Big(\frac{\ov{\xi_1}}{\mu_1}f_1^{(2,2)}(z)
+\frac{\ov{\xi_2}}{\mu_2}f_2^{(2,2)}(z)\Big)+\frac{4}{\mu_1+\mu_2}\Big(\sqrt{\frac{\mu_1}{\mu_2}}
\ov{z_1}\ov{\xi_2}-\sqrt{\frac{\mu_2}{\mu_1}}\ov{z_2}\ov{\xi_1}\Big)\\
&\cdot
\Big(\sqrt{\frac{\mu_1}{\mu_2}}z_1f_2^{(2,2)}(z)-\sqrt{\frac{\mu_2}{\mu_1}}z_2f_1^{(2,2)}(z)\Big).
\end{split}\end{equation}
Notice that
\begin{equation}\begin{split}
\ov{{\Phi}_0^{(3,0)}(z)}\cdot
2i(\frac{\eta^*_1}{\mu_1}e_1+\frac{\eta_2^*}{\mu_2}e_2)
=2i(\frac{\eta^*_1}{\mu_1}\ov{\eta_1}+\frac{\eta^*_2}{\mu_2}\ov{\eta_2}).
\end{split}\end{equation}
Hence
\begin{equation}\begin{split}\label{97eq2}
&\ov{{\Phi}_0^{(3,0)}(z)}{\Phi}_0^{(3,1)}(z)\\
=&2i\big(\frac{\eta^*_1}{\mu_1}\ov{\eta_1}
+\frac{\eta^*_2}{\mu_2}\ov{\eta_2}\big)
-4|z|^2\Big(\frac{\ov{\xi_1}}{\mu_1}f_1^{(2,2)}(z)+\frac{\ov{\xi_2}}{\mu_2}f_2^{(2,2)}(z)\Big)\\
&+\frac{4}{\mu_1+\mu_2}\Big(\sqrt{\frac{\mu_1}{\mu_2}}
\ov{z_1}\ov{\xi_2}-\sqrt{\frac{\mu_2}{\mu_1}}\ov{z_2}\ov{\xi_1}\Big)\cdot
         \Big(\sqrt{\frac{\mu_1}{\mu_2}}z_1f_2^{(2,2)}(z)-\sqrt{\frac{\mu_2}{\mu_1}}z_2f_1^{(2,2)}(z)\Big).
\end{split}\end{equation}

Combining (\ref{92eq5}) with (\ref{97eq10}) and making use of Lemma
\ref{huanglemma}  and Corollary \ref{hjy1} again, we get
\begin{equation}\begin{split}
4|z|^2\ov{\Phi_0^{(1,2)}(z)}\cdot
\Phi_0^{(2,0)}(z)+2i\ov{\phi^{(2,1)}(z)}\cdot
\phi^{(3,0)}(z)=4(\frac{1}{\mu_1}\ov{\xi_1}\eta^*_1+\frac{1}{\mu_2}\ov{\xi_2}\eta^*_2).
\end{split}\end{equation}
Namely, we have
\begin{equation}\begin{split}\label{97eq12}
|z|^2A(z,\ov{z})+2i\ov{\w{\phi}^{(2,1)}(z)}\cdot \phi^{(3,0)}(z)=0.
\end{split}\end{equation}
Here, as before, we write $A(z,\ov{z})$ for a real analytic function
which may be different in different contexts.

 Combining (\ref{97eq11}) with (\ref{97eq12})  and making use of
 Lemma \ref{huanglemma} and Corollary 3.4, we  get
\begin{equation}\begin{split}\label{97eq14}
\w{\phi}_{33}^{(2,1)}(z)=\frac{-2}{\sqrt{\mu_1+\mu_2}}\Big(\sqrt{\frac{\mu_1}{\mu_2}}z_1f_2^{(1,2)}(z)
     -\sqrt{\frac{\mu_2}{\mu_1}}z_2f_1^{(1,2)}(z)\Big).
\end{split}\end{equation}
 Next we will prove that
$\phi_{3j}^{(4,0)}(z)=0,\ \w{\phi}_{3j}^{(2,1)}(z)=0$ for
$j=4,\cdots, K$ with $K=N-n-(n-1)-(n-2)$.

\medskip

Considering terms of weighted degree $8$ in the basic equation
(\ref{*}), we get
\begin{equation}\begin{split}
&2\text{Re}\Big\{\ov{z}f^{(7)}(z,w)+\ov{f^{(3)}(z,w)}f^{(5)}(z,w)+\ov{\Phi_0^{(2)}(z,w)}{\Phi_0^{(6)}(z,w)}\\
&+\ov{\phi^{(3)}(z,w)}{\phi^{(5)}(z,w)}\Big\}
+\big|f^{(4)}(z,w)\big|^2+\big|\phi^{(4)}(z,w)\big|^2=0
\end{split}\end{equation}
over Im$(w)=|z|^2$. Namely, we have
\begin{equation}\begin{split}
&2\text{Re}\Big\{\ov{z}\Big(f^{(5,1)}(z)(u+i|z|^2)+f^{(3,2)}(z)(u+i|z|^2)^2+f^{(1,3)}(z)(u+i|z|^2)^3\Big)\\
&+\ov{f^{(1,1)}(z)(u+i|z|^2)}\cdot\Big(f^{(3,1)}(z)(u+i|z|^2)+f^{(1,2)}(z)(u+i|z|^2)^2\Big)\\
&+\ov{\Phi_0^{(2,0)}(z)}\cdot\Big(\Phi_0^{(6,0)}(z)+\Phi_0^{(4,1)}(z)(u+i|z|^2)+\Phi_0^{(2,2)}(z)(u+i|z|^2)^2\Big)\\
&+\Big(\ov{\phi^{(3,0)}(z)}+\ov{\phi^{(1,1)}(z)(u+i|z|^2)}\Big)\cdot
\Big(\phi^{(5,0)}(z)+\phi^{(3,1)}(z)(u+i|z|^2)+\phi^{(1,2)}(z)(u+i|z|^2)^2\Big)\Big\}\\
&+\Big|f^{(2,1)}(z)(u+i|z|^2)\Big|^2+\Big|\phi^{(4,0)}(z)+\phi^{(2,1)}(z)(u+i|z|^2)\Big|^2=0.
\end{split}\end{equation}
{Here we notice that the $f^{(7,0)}(z)$ term is not involved (cf.
\cite{HJX1}, Lemma 2.3(A)).} Collecting terms of the form
$z^{\a}\ov{z^{\b}}$ with $|\a|=4,|\b|=4$, we get
\begin{equation}\begin{split}\label{97eq3}
&2\text{Re}\Big\{-i\ov{z}f^{(1,3)}(z)|z|^6-\ov{f^{(1,1)}(z)}(-i|z|^2)f^{(1,2)}(z)|z|^4
+\ov{\Phi_0^{(2,0)}(z)}\Phi_0^{(2,2)}(z)(-|z|^4)\\&+
\ov{\phi^{(3,0)}(z)}\phi^{(3,1)}(z)i|z|^2+\ov{\phi^{(1,1)}(z)}(-i|z|^2)\phi^{(1,2)}(z)(-|z|^4)\Big\}\\
&+\big|\phi^{(4,0)}(z)\big|^2+\big(\big|f^{(2,1)}(z)\big|^2+\big|\phi^{(2,1)}(z)\big|^2\big)|z|^4=0.
\end{split}\end{equation}
Collecting terms of the form $z^{\a}\ov{z^{\b}}u^2$ with
$|\a|=2,|\b|=2$, we get
\begin{equation}\begin{split}\label{97eq4}
&2\text{Re}\Big\{\ov{z}f^{(1,3)}(z)3i|z|^2+\ov{f^{(1,1)}(z)}f^{(1,2)}(z)i|z|^2+\ov{\Phi_0^{(2,0)}(z)}\Phi_0^{(2,2)}(z)+
\ov{\phi^{(1,1)}(z)}\phi^{(1,2)}(z)i|z|^2\Big\}\\
&+\big(\big|f^{(2,1)}(z)\big|^2+\big|\phi^{(2,1)}(z)\big|^2\big)=0.
\end{split}\end{equation}
Collecting terms of the form $z^{\a}\ov{z^{\b}}u$ with
$|\a|=3,|\b|=3$, we get
\begin{equation}\begin{split}\label{97eq5}
&2\text{Re}\Big\{\ov{z}f^{(1,3)}(z)3(-|z|^4)+\ov{f^{(1,1)}(z)}f^{(1,2)}(z)|z|^4
+ \ov{\Phi_0^{(2,0)}(z)}\Phi_0^{(2,2)}(z)2i|z|^2
\\&+\ov{\phi^{(3,0)}(z)}\phi^{(3,1)}(z)+
\ov{\phi^{(1,1)}(z)}\phi^{(1,2)}(z)|z|^4\Big\}=0.
\end{split}\end{equation}
By (\ref{97eq3}) and (\ref{97eq4}), we get
\begin{equation}\begin{split}\label{93eq1}
&|z|^2\cdot
2\text{Re}\Big\{-4i\ov{z}f^{(1,3)}(z)|z|^4-2\ov{\Phi_0^{(2,0)}(z)}\Phi_0^{(2,2)}(z)(|z|^2)\\
&+
i\ov{\phi^{(3,0)}(z)}\phi^{(3,1)}(z)\Big\}+\big|\phi^{(4,0)}(z)\big|^2=0.
\end{split}\end{equation}
Combining this with (\ref{97eq5}), we get
\begin{equation}\begin{split}\label{98eq1}
&|z|^6A(z,\ov{z})+2|z|^2\cdot\big(
-2\ov{\Phi_0^{(2,0)}(z)}\Phi_0^{(2,2)}(z)(|z|^2)+
i\ov{\phi^{(3,0)}(z)}\phi^{(3,1)}(z)\big)+\big|\phi^{(4,0)}(z)\big|^2=0.
\end{split}\end{equation}

By (\ref{915eq1}) and Lemma \ref{product}, we get
\begin{equation}\begin{split}
\frac{1}{4}\big|\Phi_0^{(4,0)}(z)\big|^2
       =&|z|^2(\frac{1}{\mu_1}|\eta^*_1|^2
       +\frac{1}{\mu_2}|\eta^*_2|^2)-\frac{1}{\mu_1+\mu_2}
        \big|\sqrt{\frac{\mu_2}{\mu_1}}z_2\eta^*_1-\sqrt{\frac{\mu_1}{\mu_2}}z_1\eta^*_2\big|^2.
\end{split}\end{equation}
Combining this with (\ref{98eq1}) and making use of  Lemma
\ref{huanglemma}, we get
\begin{equation}\begin{split}\label{97eq6}
& |z|^4A(z,\ov{z})-4|z|^2\ov{\Phi_0^{(2,0)}(z)}\Phi_0^{(2,2)}(z)+
2i\ov{\phi^{(3,0)}(z)}\phi^{(3,1)}(z)+\frac{4}{\mu_1}|\eta^*_1|^2
       +\frac{4}{\mu_2}|\eta^*_2|^2=0,
\end{split}\end{equation}
and
\begin{equation}\begin{split}\label{99eq2}
\frac{1}{4}|\Phi_1^{(4,0)}(z)|^2=&\frac{1}{\mu_1+\mu_2}
        \big|\sqrt{\frac{\mu_2}{\mu_1}}z_2\eta^*_1-\sqrt{\frac{\mu_1}{\mu_2}}z_1\eta^*_2\big|^2.
\end{split}\end{equation}
By (\ref{99eq1}) and (\ref{99eq2}), we get
\begin{equation}\begin{split}
&{\phi}_{33}^{(4,0)}(z)=\frac{2}{\sqrt{\mu_1+\mu_2}}\Big(\sqrt{\frac{\mu_1}{\mu_2}}z_1\eta^*_2
-\sqrt{\frac{\mu_2}{\mu_1}}z_2\eta^*_1\Big),\\
&{\phi}_{3j}^{(4,0)}(z)=0\ \text{for}\ j>3.
\end{split}\end{equation}

This proves Theorem \ref{prop11} (1).

\medskip
Substituting (\ref{97eq2}) into (\ref{97eq6}), we get
\begin{equation}\begin{split}
&|z|^4A(z,\ov{z})-4|z|^2\ov{\Phi_0^{(2,0)}(z)}\Phi_0^{(2,2)}(z)+2i\ov{\Phi_1^{(3,0)}(z)}\Phi_1^{(3,1)}(z)
-8i|z|^2\big(\frac{\ov{\xi_1}}{\mu_1}f_1^{(2,2)}(z)+\frac{\ov{\xi_2}}{\mu_2}f_2^{(2,2)}(z)\big)\\
&+\frac{8i}{\mu_1+\mu_2}\big(\sqrt{\frac{\mu_1}{\mu_2}}
\ov{z_1}\ov{\xi_2}-\sqrt{\frac{\mu_2}{\mu_1}}\ov{z_2}\ov{\xi_1}\big)\cdot
         \big(\sqrt{\frac{\mu_1}{\mu_2}}z_1f_2^{(2,2)}(z)-\sqrt{\frac{\mu_2}{\mu_1}}z_2f_1^{(2,2)}(z)\big)\\ &
         +\frac{4}{\mu_1}\eta_1^*\ov{(\eta_1^*-\eta_1)}+\frac{4}{\mu_2}\eta_2^*\ov{(\eta_2^*-\eta_2)}=0.
\end{split}\end{equation}
Notice that $\Phi_1^{(3,0)}(z)=(\phi_{33}^{(3,0)}(z),0,\cdots,0)$
and $n\ge 7$. Making use of  Lemma \ref{huanglemma}, we get
\begin{equation}\begin{split}
\label{2013 4.59} &\ov{\Phi_0^{(2,0)}(z)}\Phi_0^{(2,2)}(z)=
-2i\big(\frac{\ov{\xi_1}}{\mu_1}f_1^{(2,2)}(z)+\frac{\ov{\xi_2}}{\mu_2}f_2^{(2,2)}(z)\big)+|z|^2A(z,\ov{z}).
\end{split}\end{equation}
By (\ref{212}), we have
\begin{equation}\begin{split}
f_1^{(2,2)}(z)=&\frac{i}{2}\mu_1f_1^{(2,1)}(z)-\ov{\Phi_0^{(I_1+2I_n)}}\Phi_0^{(2,0)}(z)-\ov{e_1^*}\phi^{(2,1)}(z),\\
f_2^{(2,2)}(z)=&\frac{i}{2}\mu_2f_2^{(2,1)}(z)-\ov{\Phi_0^{(I_2+2I_n)}}\Phi_0^{(2,0)}(z)-\ov{e_2^*}\phi^{(2,1)}(z).
\end{split}\end{equation}
Thus we get
\begin{equation}\begin{split}\label{i}
&2\text{Re}\Big\{-2i\big(\frac{\ov{\xi_1}}{\mu_1}f_1^{(2,2)}(z)+\frac{\ov{\xi_2}}{\mu_2}f_2^{(2,2)}(z)\big)\Big\}=
I+II+III.
\end{split}\end{equation}
Here
\begin{equation}\begin{split}\label{ii}
I&=2\text{Re}\Big\{-2i\Big(\frac{\ov{\xi_1}}{\mu_1}\frac{i}{2}\mu_1(-\xi_1)
    +\frac{\ov{\xi_2}}{\mu_2}\frac{i}{2}\mu_2(-\xi_2)\Big)\Big\}=-2(|\xi_1|^2+|\xi_2|^2).\\
II&=2\text{Re}\Big(2i\frac{\ov{\xi_1}}{\mu_1}\ov{\Phi_0^{(I_1+2I_n)}}\Phi_0^{(2,0)}(z)
     +2i\frac{\ov{\xi_2}}{\mu_2}\ov{\Phi_0^{(I_2+2I_n)}}\Phi_0^{(2,0)}(z)\Big)\\
     &=(|\xi_1|^2+|\xi_2|^2)+4\big|\frac{\xi_1}{\mu_1}e^*_1+\frac{\xi_2}{\mu_2}e^*_2\big|^2.\\
III&=2\text{Re}\Big(2i\frac{\ov{\xi_1}}{\mu_1}\ov{e^*_1}\phi^{(2,1)}(z)
     +2i\frac{\ov{\xi_2}}{\mu_2}\ov{e^*_2}\phi^{(2,1)}(z)\Big).
\end{split}\end{equation}
The equality for $I$ follows from (\ref{hh}) and  $II$ follows from
(\ref{97eq7}). By (\ref{97eq4}), we get
\begin{equation}\begin{split}\label{phi21}
&|z|^2A(z,\ov{z})+2\text{Re}\big(\ov{\Phi_0^{(2,0)}(z)}\Phi_0^{(2,2)}(z)\big)
+\big(\big|f^{(2,1)}(z)\big|^2+\big|\phi^{(2,1)}(z)\big|^2\big)=0.
\end{split}\end{equation}
Substituting  (\ref{hh}), (\ref{2013 4.59}), (\ref{i}) and
(\ref{ii}) into (\ref{phi21}), we get
\begin{equation}\begin{split}
&|z|^2A(z,\ov{z})-2(|\xi_1|^2+|\xi_2|^2)+(|\xi_1|^2+|\xi_2|^2)
+4\big|\frac{\xi_1}{\mu_1}e^*_1+\frac{\xi_2}{\mu_2}e^*_2\big|^2\\
&+2\text{Re}\Big(2i\frac{\ov{\xi_1}}{\mu_1}\ov{e^*_1}\phi^{(2,1)}(z)
     +2i\frac{\ov{\xi_2}}{\mu_2}\ov{e^*_2}\phi^{(2,1)}(z)\Big)+(|\xi_1|^2+|\xi_2|^2)+\big|\phi^{(2,1)}(z)\big|^2=0.
\end{split}\end{equation}
Hence we get
\begin{equation}\begin{split}\label{97eq9}
&|z|^2A(z,\ov{z})+\Big|\phi^{(2,1)}\text{{$(z)$}}
-2i\big(\frac{\xi_1}{\mu_1}e^*_1+\frac{\xi_2}{\mu_2}e^*_2\big)\Big|^2=0.
\end{split}\end{equation}
Substituting (\ref{97eq8}) into (\ref{97eq9}), we get
\begin{equation}\begin{split}
&|z|^2A(z,\ov{z})+\Big|\w{\Phi}_1^{(2,1)}(z)\Big|^2
-\frac{4}{\mu_1+\mu_2}\Big|\sqrt{\frac{\mu_1}{\mu_2}}z_1f_2^{(1,2)}(z)
-\sqrt{\frac{\mu_2}{\mu_1}}z_2f_1^{(1,2)}(z)\Big|^2=0.
\end{split}\end{equation}
Making use of (\ref{97eq14}) and  Lemma \ref{huanglemma}, we get
\begin{equation}\begin{split}
&\w{\phi}_{33}^{(2,1)}(z)= \frac{-2}{\sqrt{\mu_1+\mu_2}}
\Big(\sqrt{\frac{\mu_1}{\mu_2}}z_1f_2^{(1,2)}(z)
-\sqrt{\frac{\mu_2}{\mu_1}}z_2f_1^{(1,2)}(z)\Big),\\
&\w{\phi}_{3j}^{(2,1)}(z)=0\ \text{for}\ j>3.
\end{split}\end{equation}
By (\ref{2013 4.20}), the proof of Theorem \ref{prop11} (2) is also complete.

\section{Proof of Theorem \ref{thm1}}
{\bf Step (I): An application of a normal form in [HJX1] for maps
with geometric rank 1}:
  We first consider $F\in Rat({\BB}^n,{\BB}^N)$ with
geometric rank $1$. Then by Theorem 1.2 of [HJX1], $F$ is equivalent
to a map of the form $\Phi=(z_1,\cdots,
z_{n-1},z_nH(z)):=(\phi_1,\cdots,\phi_N)$ with $H\in
Rat({\BB}^n,{\BB}^{N-n+1})$ also of geometric rank one. We first
have the following:

\begin{lemma}\label{claim}
 If $H({\BB}^n)$ is contained in an affine subspace of
dimension $m$ in ${\CC}^{N-n+1}$, then $F({\BB}^n)$ is contained in
an affine subspace of dimension $m+n$ in ${\CC}^{N}$.
\end{lemma}

\begin{proof}
Indeed, we first notice that linear fractional transformations map
affine linear subspaces to affine linear subspaces. Also,
$F({\BB}^n)$ is contained in an affine subspace of dimension $m$, if
and only if $F$ is equivalent to a map of the form $(G,0)$ with $G$
having $m$-components.
  Now suppose the image of $H=(h_1,\cdots,
h_{N-n+1})$ is contained in an affine subspace of dimension $m\le N-n$,
then there are $(N-n-m+1)$ linearly independent vectors
$\mu_j=(a_{j1},\cdots,a_{jk})$ with $k=N-n+1$ such that
$\sum_{l=1}^{k}a_{jl}h_l(z)\equiv c_j$ for certain $c_j \in {\CC}$.
If $c_j=0$ for all $j$, then
$\sum_{l=1}^{k}a_{jl}\phi_{n-1+l}(z)\equiv 0$. Hence $\Phi({\BB}^n)$
is contained in an affine linear subspace of dimension $m+n-1$.
Otherwise, assume without loss of generality that $c_1=1$. Then we
have $\sum_{l=1}^{k}(a_{jl}- c_j a_{1l})\phi_{n-1+l}(z)\equiv 0$.
Notice that
$\{\mu_2-c_2\mu_1,\cdots,\mu_{N-n-m+1}-c_{N-n-m+1}\mu_1\}$ is also
linearly independent, we see that  $\Phi({\BB}^n)$ is contained in
an affine subspace of dimension $m+n$. This proves Lemma
\ref{claim}.
\end{proof}

\medskip
Applying the gap rigidity in [HJX1] {to $H$} and Lemma \ref{claim},
we  see that when $3n+1\le N\le 4n-5$ and $n\ge 6$, $F({\BB}^n)$ is
contained in an affine linear subspace of dimension $3n$.
This proves
Theorem \ref{thm1} in case the map $F$ has geometric rank one.

Indeed, by an induction argument, we see that when
$N<(k+1)n-\frac{k(k+1)}{2}$, $F({\BB}^n)$ is contained in a linear
affine subspace of dimension $kn$, if $(k+1)n-\frac{k(k+1)}{2}>0$.

\medskip
{\bf  Step (II):  Completion of the Proofs of Theorems \ref{thm1}}:
{By} Lemma 3.2 in [Hu2], when $N\le 4n-7$, any $F\in
Rat({\BB}^n,{\BB}^{N})$ can only have geometric rank $\kappa_0=0,1$,
{or} $2.$ When $\kappa_0=0$, $F$ is linear and thus Theorem
\ref{thm1} follows trivially. When $\kappa_0=1$, the proof of
Theorem \ref{thm1} is already done in Step (I). The case of Theorem
\ref{thm1} for maps with geometric rank two is obviously a special
case of the following Theorem \ref{thm2}.

\begin{theorem} \label {thm2} Let $F$ be a proper rational map from ${\HH}_n$ into ${\HH}_N$
with geometric rank $\kappa_0=2$. Assume that $n\ge 7$ and $3n\le
N\le 4n-6$. Then $F$ is equivalent to a map of the form $(G,0')$
where $G$ is a proper rational   map from ${\HH}_n$ into
${\HH}_{N'}$ with $N'=3n$.
\end{theorem}

{Since $F$ is rational, by a result of Cima-Suffridge, the above $F$
extends holomorphically across $\p {\HH}_n$.}

\medskip

Let $N$ be such that $N\le 4n-6$. Let $F$ be a proper rational
holomorphic map from ${\HH}_n$ into ${\HH}_N$ with geometric rank
$\kappa_0=2$ and $F(0)=0$.
 As mentioned {in} $\S 2$, we can assume, without loss
of generality, that $F$ satisfies the normalization in Theorem
\ref{thm3}.

Write ${\cal L}_j=\frac{\partial}{\partial
z_j}-2i\overline{z_j}\frac{\partial}{\partial w}$ for $j=1,\cdots,
n-1,$ which form  a basis of tangent vector fields of type $(1,0)$
along $\partial {\HH}_n$. Let ${\cal L}^\a$ be defined in the
standard way. Notice that for any smooth function $h$ near $0$,
$\mathcal{L}^{\a}h|_0=\frac{\p ^{|\a|}}{\p z^\a}h|_0:=D^\a_zh|_0$.

Assume the normalization in Corollary \ref{hjy1} for $F$. Also
assume that $\varphi_{33}^{(3,0)}(z)\not \equiv 0$. Then
\begin{equation}
\underset{|\beta|\le 3}{span}\{ {\cal L}^\beta F|_0\}= {span}\{ (0,
...,0, 1^{j^{th}},  0, ...,0),\ 1\leq j \leq n+\sharp
\mathcal{S}_0=n+(n-1)+(n-2)=3n-3\}.
\end{equation}
Applying Theorem \ref{prop11} (1), we see that
\begin{equation}
\underset{|\beta|\le 4}{span}\{ {\cal L}^\beta F|_0\}= {span}\{ (0,
...,0, 1^{j^{th}},  0, ...,0),\ 1\leq j \leq n+\sharp
\mathcal{S}_0\}.
\end{equation}
Hence
\begin{equation}
\underset{|\beta|\le 4}{span}\{ {\cal L}^\beta F|_0\}=
\underset{|\beta|\le 3}{span}\{ {\cal L}^\beta F|_0\}.
\end{equation}

Now, we  proceed  in  a similar  way as in \cite{Hu1}, though the
situation in \cite{Hu1} is harder for the maps there are only
assumed to be twice differentiable. Notice that we have assumed that
 $\phi^{(3,0)}_{33}(z)\not \equiv 0$.

For any $p\in \HH_n$ $(\approx 0)$, there exist $\tau_p\in
Aut_0({\HH_N} $), $\sigma_p \in Aut_0(\HH_n)$ such that $G_p=\tau_p
\circ F_p \circ \sigma_p$ satisfies the normalization condition  in
Theorem \ref{thm3}.  Moreover, we can also have the
$\Phi^{(3,0)}_1(z)$ coming from $G_p$ not identically zero,
for we can choose $\tau_p,\sigma_p$ to depend smoothly on $p$.
Hence, after applying $U_p$, a unitary matrix transformation, to
normalize the $\Phi_1$-part, we get the normalization as in
Corollary \ref{hjy1} for the new map  with the corresponding
$\phi^{(3,0)}_{33}(z)\not \equiv 0$.
 Notice that for the new $G_p$, we have
\begin{equation}\begin{split} \label{0011}
\underset{|\beta|\le 4}{span}\{ {\cal L}^\beta G_p|_0\}=
\underset{|\beta|\le 3}{span}\{ {\cal L}^\beta G_p|_0\}, \
\hbox{or}\ \underset{|\beta|\le 4}{span}\{ {D}_z^\beta G_p|_0\}=
\underset{|\beta|\le 3}{span}\{ {D}_z^\beta G_p|_0\}.
\end{split}\end{equation} Also the dimension of the above space is $n+\sharp
\mathcal{S}_0$ for any $p\approx 0$.

 Still write $\tau_p$ for $U_p\circ \tau_p$. Then
$F_p=\tau_p^{-1}\circ G_p \circ \sigma_p^{-1}$. Now, for any
$|\alpha|=4$, we claim  that
\begin{equation}
\label{(5)} {D}_z^\alpha(\tau_p^{-1}\circ G_p \circ
\sigma_p^{-1})|_0 \in \underset{|\beta|\le 3}{span} \{
{D}_z^\beta(\tau_p^{-1}\circ G_p\circ \sigma_p^{-1})|_0\}, \
\hbox{or}\  {\cal L}^\alpha F_p|_0 \in \underset{|\beta|\le 3}{span}
\{ {\cal L}^\beta F_p|_0\}.
\end{equation}
Here, as defined before, $D_z^\alpha$ is the regular
differentiation, with respect to $z$, of order $|\alpha|$.

Indeed, write
$$
\sigma_p^{-1}=\Big(\mu
\frac{z-aw}{q(z,w)}A,\mu^2\frac{w}{q(z,w)}\Big),\
\tau_p^{-1}=\Big(\w{\mu}
\frac{\w{z}-\w{a}\w{w}}{\w{q}(\w{z},\w{w})}\w{A},\w{\mu}^2\frac{\w{w}}{\w{q}(\w{z},\w{w})}\Big)
$$
with $\mu,\w{\mu}\neq 0$, $A,\w{A}$ unitary matrices,
$q(0),\w{q}(0)=1$.

Write $G_p=(h(z,w),w)$. Then
\begin{equation}\label{fpgp}
F_p(z,0)=\big(\frac{\w{\mu}}{{q^*}(z)}h(\frac{\mu
z}{q(z,0)}A,0)\w{A},0\big),
\end{equation}
for a certain holomorphic
function $q^*(z)$ with $q^*(0)=1$. Now to show that for any
$|\a|=4$, $D^\a_zF_p(z,0)|_0\in \underset{|\beta|\le
3}{span}\{D_z^\beta F_p(z,0)|_0\}$, it suffices to show that
$$
D^\a_z h\big(\frac{\mu z}{q(z,0)}A,0\big)\big|_0\in
\underset{|\beta|\le 3}{span}\big\{D_z^\beta h\big(\frac{\mu
z}{q(z,0)}A,0\big)\big|_0\big\}.
$$
Notice that
$$
\underset{|\a|\leq k}{span}\big\{D^\a_z h\big(\frac{\mu
z}{q(z,0)}A,0\big)\big|_0 \big\}=\underset{|\a|\leq
k}{span}\big\{D^\a_z h(z,0)\big|_0 \big\}
$$
and notice that (by (\ref{0011})) $$\
\underset{|\a|\leq 4}{span}\{D^\a_z h( z,0)|_0 \}=\underset{|\a|\leq
3}{span}\{D^\a_z h(z,0)|_0 \}.
$$ We conclude that $$\underset{|\a|\leq 4}{span}\{D^\a_z F_p(z,0)|_0
\}=\underset{|\a|\leq 3}{span}\{D^\a_z F_p(z,0)|_0 \}.$$ We thus
arrive at a  proof for  the claim. Moreover, we also conclude from
(\ref{fpgp}) that
$$dim\Big(\underset{|\a|\leq 3}{span}\{D^\a_z
F_p(z,0)|_0\}\Big)=dim\Big(\underset{|\a|\leq 3}{span}\{D^\a_z
G_p(z,0)|_0\}\Big)=n+\# {\cal S}_0.$$

\medskip

Since ${\cal L}^\a(F_p)|_0={\cal L}^\a(F)(p)$, we get  that for
$|\a|=4$,  $ {\cal L}^\alpha F(p) \in \underset{|\beta|\le 3}{span}
\{ {\cal L}^\beta F(p)\}.$ Since $\underset{|\beta|\le 3}{span} \{
{\cal L}^\beta F(p)\} $ has a fixed dimension $n+\# {\cal S}_0$ for
$p\approx 0$, we can write, for any $\a$, ${\cal L}^\alpha F(p)$ as
a smooth linear combination of a fixed (smoothly varied) basis from
$\underset{|\beta|\le 3}{span} \{ {\cal L}^\beta F(p)\} $.
Successively applying $\overline{\cal L}_j$, ${\cal L}_k$ as in the
proof of [Lemma 4.3, Hu1] to the so obtained expressions and using
the bracket property for such vector fields, we can obtain as in
[Hu1] that $D^\alpha F(0) \in \underset{|\beta|\le 3}{span}\{D^\beta
F(0)\}$ for any multiple index $\alpha$. Here $D^\alpha$ is the
regular total differentiation (not just along the $z$-directions) of
order $|\alpha|$. Thus $F(z,w) \in \underset{|\beta|\le
3}{span}\{D^\beta F(0)\}$ by the Taylor expansion for $(z,w)\approx
0$. Now,
write as before, {$\phi^{(1,1)}$}$(z)w=({e}^*_1 z_1 + {e}^*_{2}
z_{2})w.$ By Theorem \ref{prop11} (2) (3), we see that $
\underset{\beta\le 3}{span}\{D^\beta F(0)\}$ stays in the span of
the following vectors: $$\Big\{(0, ...,0, 1^{j^{th}},  0, ...,0),
(0, .., 0, 1),\ (0, 0, ...,0, \hat{e}_1,0),\\
(0,\cdots, 0, ...,0,\hat{e}_{2},0)\Big\},$$ where $\ 1\le j\le
\big((n-1)+(n-1)+(n-2)\big)+1$. Hence $F(\HH_n)$ is contained in a
linear subspace with dimension equal to $3n-3+2+1=3n$.
Hence, we see the  the proof of Theorem \ref{thm1} in this setting.

Now, if, for a certain $p_0\approx 0$, the $\phi^{(3,0)}_{33}(z)$
associated with $F_{p_0}$ is  not a zero polynomial, then we can
consider $F_{p_0}$ instead of $F$ and apply the above argument to
conclude the proof of Theorem \ref{thm2}. Finally, if after the
normalization of  $F_p$ to the form as in Corollary \ref{hjy1} for
any $p\approx 0$, we have $\Phi_1^{(3,0)}(z)\equiv 0$, then a
similar method as above shows that $ {\cal L}^\alpha F(p) \in
\underset{|\beta|\le 2}{span} \{ {\cal L}^\beta F(p)\}$ with
$dim\Big(\underset{|\beta|\le 2}{span} \{ {\cal L}^\beta
F(p)\}\Big)\equiv n+\#{\cal S}_0-1.$ Hence, $F({\BB}^n)$ is
contained in a complex linear subspace of dimension
$$n+(n-1)+(n-2)+2=3n-1,$$ spanned by
 $$\bigg\{(0, ...,0, 1^{j^{th}},  0, ...,0),
(0, .., 0, 1),\ (0, 0, ...,0, \hat{e}_1,0),\\
(0,\cdots, 0, ...,0,\hat{e}_{2},0)\bigg\},$$ where $\ 1\le j\le
(n-1)+(n-1)+(n-2)$. The proof of Theorem \ref{thm1} is complete now.
 \ \ \ $\Box$

\bigskip
{\bf Remark A}:  Consider the  map defined in (\ref{example}):

\medskip

$F=\big(z_1,\cdots,z_{n-2}, \lambda z_{n-1},z_n, \sqrt{1-\lambda^2}
z_{n-1}(z_1,\cdots,z_{n-1},\mu z_n, \sqrt{1-\mu^2} z_nz)\big),\ \
\lambda, \mu \in (0,1).$
\medskip
The map is apparently a proper monomial map from ${\BB}^n$ into
${\BB}^{3n}$. Write $F=(f_1,\cdots,f_{3n})$. We claim that $F$ is
not equivalent to a map of the form $(G,0)$. Otherwise, there are
complex numbers $\{a_j\}_{j=1}^{3n}$, not all zeros, such that
$\sum_{j=1}^{3n}a_j f_j\equiv 0$, which is obviously impossible just
by comparing the coefficients of degree $3,2,1$, respectively.

The map $F$ is of degree three. It has   geometric rank two just by
observing that the largest dimension of the  affine subspaces where
$F$ is linear is of codimension two. (By a result in [Hu2], this
codimension is the same as the geometric rank of the map.)

As we discussed above, the span of the first and the second jets has
dimension $3n-1$. That means we have one more independent element
from the third jet.
Hence $\phi^{(3,0)}_{33}\not \equiv 0$ for such a map (after
transforming to the Heisenberg hypersurface and after the
normalization) in a generic position.

\bigskip
{\bf Remark B}:  We mention that even for  $N\ge 3n-2$, there are
many rational proper holomorphic map from ${\BB}^n$ into ${\BB}^N$
that are not equivalent to any polynomial maps as shown in a paper
by Faran-Huang-Ji-Zhang \cite{FHJZ}. The following is one of the examples
provided in \cite{FHJZ}:

\medskip
 Let $ F(z',z_n)=\bigg(z', z_nz', z_n^2(\frac{\sqrt{1-|a|^2}z'}{1- \ov a
z_n}, \frac{z_n - a}{1- \ov a z_n}) \bigg)$ with $|a|<1$, which is a
proper rational holomorphic  map  from $\BB^n$ into  $\BB^{3n-2}$.
Then $F$ has geometric rank $1$ and is linear along each hyperplane
defined by $z_n = constant$. $F$ is   equivalent to a proper
polynomial map from ${\BB}^n$ into  ${\BB}^{3n-2}$
 if and only if $a = 0$.

\medskip
This example gives an indication that it is  unpractical  to achieve
a precise classification for  proper rational proper  maps from
${\BB}^n$ into ${\BB}^N$ with $N\in {\cal I}_3$ to get the gap
rigidity.

\bigskip
{\bf Remark C}: This paper is a  simplified  version of the authors'
early   preprint. Theorem \ref{thm1}   was first announced in \cite{HJY}
(Theorem 2.9 in \cite{HJY}).

\bibliographystyle{amsalpha}

\end{document}